\documentclass[11pt]{amsart}
\usepackage{hyperref}
\usepackage[all]{xy}
\usepackage{tikz}
\usepackage{tikz-cd}
\usetikzlibrary{decorations.pathmorphing}
\usepackage{etex}
\usepackage[usenames,dvipsnames]{pstricks}
\usepackage{epsfig}
\usepackage{graphicx,color}
\usepackage{geometry}
\usepackage{amssymb}
\usepackage{cite}
\usepackage{fullpage}
\xyoption{poly}
\usepackage{url}
\numberwithin{equation}{section}
\numberwithin{figure}{section}
\usepackage{longtable}

\newtheorem{theorem}{Theorem}[section]
\newtheorem{lemma}[theorem]{Lemma}
\newtheorem{proposition}[theorem]{Proposition}
\newtheorem{corollary}[theorem]{Corollary}
\theoremstyle{definition}
\newtheorem{definition}[theorem]{Definition}

\newtheorem{remark}[theorem]{Remark}

\newtheorem{example}[theorem]{Example}

\newtheorem*{question*}{Question}

\newtheorem*{steps*}{Answer/steps}

\newtheorem*{progress*}{Progress}
\newtheorem*{example*}{Example}

\newtheorem*{remark*}{Remark}
\newtheorem*{remarks*}{Remarks}
\newtheorem*{definition*}{Definition}
\usepackage{comment}
\usepackage{calrsfs}
\usepackage{stmaryrd}
\usepackage{mathrsfs}
\usepackage{bm}
\usepackage[utf8]{inputenc}
\usepackage[OT2,T1]{fontenc}
\usepackage[english]{babel}
\usepackage{textcomp}
\usepackage{times}
\usepackage[scaled=0.92]{helvet}

\newcommand{\CC}{\mathbb{C}}
\newcommand{\QQ}{\mathbb{Q}}

\newcommand{\RR}{\mathbb{R}}
\newcommand{\ZZ}{\mathbb{Z}}

\newcommand{\FF}{\mathbb{F}}
\newcommand{\PP}{\mathbb{P}}

\newcommand{\pp}{\mathfrak{p}}

\newcommand{\pP}{\mathfrak{P}}
\newcommand{\pQ}{\mathfrak{Q}}

\newcommand{\kbar}{\overline{k}}
\newcommand{\rbar}{\overline{r}}
\newcommand{\GGm}{\mathbb{G}_m}
\newcommand{\GGa}{\mathbb{G}_a}
\newcommand{\mmu}{\mu}

\DeclareMathOperator{\Aut}{Aut}

\DeclareMathOperator{\Gal}{Gal}

\DeclareMathOperator{\cha}{char}
\DeclareMathOperator{\Jac}{Jac}
\DeclareMathOperator{\Hom}{Hom}

\newcommand{\Nm}{\mathrm{Nm}^{q}_{k}}
\def\phi{\varphi}
\def\theta{\vartheta}

\begin{document}

\title{Cubic function fields with prescribed ramification}

\author[Karemaker]{Valentijn Karemaker}
\address{
  Valentijn Karemaker,
  Mathematical Institute, Utrecht University,
  P.O. Box 80010, 3508 TA Utrecht, the Netherlands
}
\email{V.Z.Karemaker@uu.nl}

\author[Marques]{Sophie Marques}
\address{
  Sophie Marques,
  Mathematics Division,
  Stellenbosch University,
  Mathematics/Industrial Psychology Building,
  Merriman Street,
  7600 Stellenbosch,
  South Africa
}
\email{smarques@sun.ac.za}

\author[Sijsling]{Jeroen Sijsling}
\address{
  Jeroen Sijsling,
 Institut für Algebra und Zahlentheorie, 
  Universität Ulm,
  Helmholtzstrasse 18,
  89081 Ulm,
  Germany
}
\email{jeroen.sijsling@uni-ulm.de}


\begin{abstract}
  This article describes cubic function fields $L / K$ with prescribed ramification, where $K$ is a rational function field. We give general equations for such extensions, an explicit procedure to obtain a defining equation when the purely cubic closure $K' / K$ of $L / K$ is of genus zero, and a description of the twists of $L / K$ up to isomorphism over~$K$. For cubic function fields of genus at most one, we also describe the twists and isomorphism classes obtained when one allows Möbius transformations on $K$. The article concludes by studying the more general case of covers of elliptic and hyperelliptic curves that are ramified above exactly one point.
\end{abstract}

\maketitle

\section*{Introduction}

Let $k$ be a perfect field, let $K$ be a function field over $k$, and let $L/K$ be a (geometric and separable) cubic extension of $K$. If $K = k (x)$, then a cubic extension of $K$ is often also called a \emph{cubic function field}. Such fields have been studied extensively in the literature, which has led to several classifications and descriptions of these extensions and their minimal polynomials, especially when $k$ is a finite field. For instance, cubic (and occasionally more general dihedral) function fields of fixed (resp.~bounded) discriminant are constructed and counted in \cite{pohst1, pohst2} (using class field theory), \cite{weir} (using an algorithmic implementation of Kummer theory, applicable to a larger class of field extensions), \cite{rozenhart1, rozenhart2} (building on the algorithm of \cite{belabas}), and \cite{scheidler1, scheidler2} (building on results of Shanks for cubic number fields, see e.g.~\cite{shanks-cubic}). In addition, \cite{scheidler3, scheidler4} also systematically study the signature of cubic function fields, and normal forms for their minimal polynomials.

In the present paper we take an approach that classifies extensions with a given set of ramified places instead of those with a given discriminant. As a first step towards such a classification we give a new and simple criterion (Theorem \ref{thm:norms}) for the existence of a cubic extension with given ramification. To phrase this criterion, instead of using the quadratic resolvent field defined by the discriminant of $L/K$ as in \cite{weir}, we use the notion of a purely cubic closure from \cite{MWcubic, MWcubic2}. This is a special instance of the more generally defined \emph{reflection field} \cite{cohenetal}. We briefly recall its definition.

When the characteristic of $k$ is not equal to $3$, the extension $L/K$ admits a generator $y$ whose minimal polynomial over $K$ is either of the form $X^3-3X-\alpha$ or of the form $X^3-\beta$ for some $\alpha, \beta \in K$. In the latter case the extension $L/K$ is called \emph{purely cubic}, and otherwise it is called \emph{impurely cubic}. The papers \cite{MWcubic, MWcubic2} determine standard forms for these extensions and show how to read off the ramification properties from these forms. Moreover, they prove that any impurely cubic extension $L/K$ admits a so-called \emph{purely cubic closure} $K'$, which is a degree two extension of $K$ such that $LK'/K'$ is purely cubic. Section~\ref{sec:notions} recalls that in characteristic not equal to $2$ the purely cubic closure can be obtained by adjoining to $K$ a square root of $-3$ times the discriminant of $L/K$. (Lemma \ref{lem:pcvsdisc}(2) gives a similar but slightly more complicated description in characteristic $2$.)

Theorem~\ref{thm:norms} now states that if $K'$ is of genus zero and $T$ is a given set of places of $K'$, then there exists a cubic extension $L/K$ with purely cubic closure $K'$ and total (that is, triple) ramification precisely at the places in $T$ if and only if all places in $T$ split in the quadratic extension $K'$. We discuss in Remark~\ref{rem:startwithS} how this allows one to find cubic extensions $L/K$ with total ramification precisely at the places in a given set $T$ and partial (that is, double) ramification precisely at the places in another specified set.

The main idea behind the existence result in Theorem~\ref{thm:norms} is to determine an extension $L/K$ by descending from its base change $LK'/K'$ to the purely cubic closure $K'$. A similar result was obtained by Weir in his thesis \cite[Theorem 4.1.12]{weir-thesis} for the more general case of dihedral function fields. We illustrate some advantages of our more geometrically inspired approach, because of which our results also go beyond those in \cite{weir-thesis}. 

Firstly, our approach is careful enough to apply over arbitrary base fields $k$ whose characteristic does not equal $3$. In particular, we do not have to avoid the case where $k$ has characteristic $2$, as \emph{loc. cit.} does.

Secondly, as recorded in Theorem~\ref{thm:unifeq} and Remark~\ref{rem:min}, given a purely cubic closure $K'$ and a total ramification locus $T$, our descent procedure also allows us to construct explicit defining polynomials for such extensions of the simple form $y^3 - 3 c y - \alpha = 0$, where $c \in k$ and where the poles and zeros of $\alpha \in K$ are of minimal order. Moreover, our technique further allows us to determine the number of such extensions up to isomorphism over $\overline{k}$ in Corollary~\ref{cor:allwithclos}. In Lemma~\ref{lem:twistisos} we obtain a simple and explicit description of the \emph{twists} of the extensions, that is, of the $k$-isomorphism classes that constitute a given $\overline{k}$-isomorphism class of cubic extensions.

Thirdly, in Section~\ref{sec:biisotwists} we go one step further in that we also consider the \emph{bi-isomorphism classes} and the corresponding notion of \emph{bi-twists} of cubic extensions $L/K$ by also allowing an isomorphism of $K$. (Precise definitions, as well as their geometric interpretation, are given in Section~\ref{sec:notions}.) We give a full classification of bi-isomorphism classes of cubic extensions $L$ of genus at most one whose purely cubic closure $K' \neq K$ is of genus zero in terms of the genus of $L$ and the ramification locus. Our description also counts the resulting number of bi-isomorphism classes in the case where the base field $k$ is finite, and shows how to find explicit defining equations for all classes. Similar counts of bi-isomorphism classes when $k$ is finite were included in \cite[\S 6.2.2]{weir-thesis}, but only in other special cases (involving function fields of genus one). Moreover, our approach has the advantage of describing all bi-isomorphism classes via explicit parametrizations.

Finally, in Section~\ref{sec:parshin} we consider an important case where $K$ itself is \emph{not} of genus zero, let alone a rational function field (an assumption which is made in all of \cite{weir-thesis}), by constructing so-called \emph{Parshin covers}. These are cubic extensions of genus one and genus two function fields $K$ over general fields $k$ with ramification at a single place of degree one. They played a role in Parshin's reduction of the Mordell conjecture to the Shafarevich conjecture and have also appeared in Lawrence--Venkatesh' new proof of this conjecture, as discussed in \cite{parshin}. The proof of our main Theorem~\ref{thm:norms} is sufficiently general to apply to this case as well and gives an explicit description of such Parshin covers. This shows the continued relevance of our alternative approach beyond the (already important) case where the purely cubic closure~$K'$ is of genus zero.

Both Sections \ref{sec:quad} and \ref{sec:pure} contain important preliminaries for our main results. Section \ref{sec:quad} considers isomorphism classes and bi-isomorphism classes of quadratic extensions and determines an explicit description of certain relevant cohomology classes in Proposition \ref{prop:Qbitwistscharne2}. Section~\ref{sec:pure} studies the simpler case where the extension $L/K$ is purely cubic (so $K'=K$); here we can apply Kummer theory to determine the extensions with a given total ramification locus $T$. Theorem \ref{thm:purelycubic} gives corresponding explicit defining equations for $L/K$ when $K = k (x)$ is a rational function field. (A similar but less explicit result holds when $K$ is merely supposed to be of genus zero: see Remark~\ref{rem:conic1}.) More precisely: the purely cubic extensions of $K$ are of the form $y^3 = u \prod_i P_i^{\pm 1}$, where the $P_i$ are polynomials defining the finite places in $T$ and where $u \in k^*$. Using a recursion, Theorem \ref{thm:purelycubic} then determines the number of $\overline{k}$-isomorphism classes of the resulting extensions. The twists of a given extension are governed by the class of the scalar $u$ in $k^* / (k^*)^3$, cf.~Proposition~\ref{prop:pctwists}. In Theorems \ref{thm:biisog0} and \ref{thm:biisog1} we consider the bi-twists of the purely cubic extensions $y^3 = x$ and $y^3 = x (x - 1)$ . Theorem \ref{thm:biisog0} shows that when the base field $k$ is a number field, there exists a bi-twist of the extension defined by $y^3 = x$ that, while of genus zero, is not a rational function field. By contrast, Theorem \ref{thm:biisog1} shows that all bi-twists of $y^3 = x (x - 1)$ have a rational base field.

\section{Fundamental Notions}\label{sec:notions}

We let $k$ be a chosen \textbf{perfect} base field, and assume that its characteristic is unequal to $3$. Let $K$ be a function field over $k$.

\begin{definition}\label{def:cubext}
  A \emph{cubic extension} of $K$ is a separable field extension $L$ of $K$ of degree three that is geometric, i.e., that is not of the form $K \ell$ with $\ell$ a degree-three extension of $k$.
\end{definition}

By virtue of restricting our extensions to those in Definition \ref{def:cubext}, we can also consider all of our objects and morphisms in the language of smooth algebraic curves over $k$. In this way, a cubic extension $L/K$ is nothing but a degree-three separable map $X \to \PP^1_k$ of curves over $k$. This is also reflected in the following definitions.

\begin{definition}
  Let $L_1/K$ and $L_2/K$ be two cubic extensions of $K$. Then $L_1/K$ and $L_2/K$ are called \emph{isomorphic} if they are isomorphic as extensions of $K$. Phrased geometrically in terms of the separable maps $X_1 \to \PP^1_k$ and $X_2 \to \PP^1_k$ corresponding to the extensions $L_1/K$ and $L_2/K$, this means that there should exist a $k$-rational map of curves $\phi$ such that the diagram
  \begin{equation}\label{eq:iso}
    \xymatrix{ X_1 \ar[dr] \ar[rr]^\phi & & X_2 \ar[dl] \\ & \PP^1_k &}
  \end{equation}
  commutes.

  The extensions $L_1/K$ and $L_2/K$ are called \emph{bi-isomorphic} if there exists a pair of field isomorphisms $\iota : L_1 \to L_2$ and $\iota_K : K \to K$ that are compatible with the inclusions of $K$ into $L_1$ and $L_2$. Phrased geometrically, this means that there should exist $k$-rational maps of curves $\phi$ and $\psi$ such that the diagram
  \begin{equation}\label{eq:biiso}
    \xymatrix{ X_1 \ar[d] \ar[r]^{\phi} & X_2 \ar[d] \\ \PP^1_k \ar[r]^{\psi} & \PP^1_k}
  \end{equation}
  commutes.
\end{definition}

\begin{remark}
  Many of the results in this paper are specific to the case where $K = k (x)$ is a rational function field, or, slightly more generally, to the case where $K$ is the function field of a non-trivial conic over $k$. Base function fields $K$ of higher genus are considered in Section \ref{sec:parshin}.
\end{remark}

\begin{definition}
  A cubic extension $L/K$ is called \emph{purely cubic} if it admits a generator with minimal polynomial $X^3 - \beta$ over $K$, and \emph{impurely cubic} otherwise.
\end{definition}

We need the following fundamental notion from \cite{MWcubic2} (but see also \cite{cohenetal}):

\begin{definition}\label{def:closure}
  For any cubic extension $L/K$, we define the \emph{purely cubic closure} of $L/K$ to be the smallest extension $K'$ of $K$ such that $LK'/K'$ is purely cubic. Note that there is indeed a unique smallest such extension of $K$ by \cite[Theorem~3.1]{MWcubic2}.
\end{definition}

\begin{theorem}[{\cite[Theorem 3.1]{MWcubic2}}]\label{thm:cubclos}
  The purely cubic closure $K'$ exists for any cubic extension $L/K$. If $L/K$ is not purely cubic and generated by an element $y$ of $L$ with minimal polynomial $X^3 - 3 X - \alpha$ over $K$, then $K' = K(s)$, where $s$ is a root of the polynomial $X^2 + \alpha X + 1$ and $s + s^{-1} = \alpha$. In particular, we always have $[K': K] \leq 2$.

  Let $L'$ be the extension of $K'$ defined by adjoining a root $w$ of the defining polynomial $w^3 = s$. Then~$L$ is the fixed field of $L'$ under the involution sending $(s, w)$ to $(s^{-1}, w^{-1})$. A generating element for~$L$ is given by the trace $y = w + \sigma (w) = w + w^{-1}$, whose minimal polynomial over $K$ is given by $X^3 -3X - (w^3 + w^{-3})$.
\end{theorem}

\begin{definition}\label{def:res}
  Let $L/K$ be a cubic extension. We define the \emph{resolvent extension} of $L/K$ to be the smallest extension $Q$ of $K$ for which $L Q / Q$ is Galois. Note that there is indeed a unique smallest such extension of $K$, namely the field corresponding to $A_3 \cap \Gal (L/K) \subset S_3$ under the Galois correspondence.
\end{definition}

\begin{remark}
  Let $L / K$ be a cubic extension, and let $M$ be an overfield of $K$. Then by \cite[Theorem 3.1]{MWcubic2} the extension $L M / M$ is purely cubic if and only if $M$ contains the purely cubic closure $K'$ of $L / K$. Similarly, $L M / M$ is Galois if and only if $M$ contains the resolvent extension $Q$ of $L / K$, since the Galois closure of $L M / M$ is obtained by taking the compositum with $Q M / M$. In other words, ``smallest'' in Definitions \ref{def:closure} and \ref{def:res} can be taken with respect to either degree or inclusion.
\end{remark}

The techniques from \cite{MWcubic} now show us the following.

\begin{lemma}\label{lem:pcvsdisc}
  \begin{enumerate}
    \item Suppose that $\cha (k) \neq 2$. Let $L/K$ be a cubic extension with discriminant $\delta$, so that the resolvent extension $Q$ of $L/K$ equals $K (\sqrt{\delta})$. Then the purely cubic closure $K'$ of $L/K$ is given by $K' = K (\sqrt{-3 \delta})$.
    \item Suppose that $\cha (k) = 2$. Let $L/K$ be a cubic extension whose resolvent extension $Q$ is defined by the polynomial $X^2 + X + \gamma$. Then the purely cubic closure is defined by the polynomial $X^2 + X + \gamma + 1$.
  \end{enumerate}
\end{lemma}

\begin{proof}
  (1) The discriminant $\delta$ of $f$ equals $-27 \alpha^2 + 108$ (cf.\ \cite[Theorem 2.3]{MWcubic}), and square root of $\delta$ cuts out the resolvent extension by the description in terms of the Galois correspondence in Definition~\ref{def:res}. On the other hand, the purely cubic closure $K' = K (s)$ is defined by the polynomial $X^2 + \alpha X + 1$ (cf. Theorem \ref{thm:cubclos}), which has discriminant $\alpha^2 - 4$.

  (2) Let $X^3 + X + \alpha$ be a minimal polynomial defining $L$ over $K$. By \cite[Theorem~2.3]{MWcubic}, the resolvent extension is defined by $X^2 + X + (1/\alpha)^2 + 1$ (alternatively, it is defined by $X^2 + X + (1/\alpha) + 1$ after changing the Artin--Schreier generator from $z$ to $z + (1/\alpha)$). On the other hand, Theorem \ref{thm:cubclos} shows that the purely cubic closure is defined by $X^2 + \alpha X + 1$. Scaling $X$ gives the defining polynomial $X^2 + X + (1/\alpha)^2$ (which under the same change of generator becomes $X^2 + X + 1/\alpha$).
\end{proof}

\begin{definition}\label{def:compext}
  Let $K$ be a field, and let $Q_1$ and $Q_2$ be two extensions of $K$ of degree at most two. We define the \emph{complementary extension} $Q_3$ to $Q_1$ and $Q_2$ over~$K$ to be the remaining entry in the (possibly degenerate) subfield lattice generated by $Q_1$ and $Q_2$ over $K$. In other words, $Q_3$ is defined as follows:
  \begin{enumerate}
    \item If $Q_1 = Q_2 = K$, then $Q_3 = K$.
    \item If $Q_1 = K$ and $Q_2 \neq K$, then $Q_3 = Q_2$.
    \item If $Q_1 \neq K$ and $Q_2 = K$, then $Q_3 = Q_1$.
    \item If $Q_1 = Q_2$ are both non-trivial extensions of $K$, then $Q_3 = K$.
    \item If $Q_1 \neq Q_2$ are both non-trivial extensions of $K$, then $Q_3$ is the unique subfield of the biquadratic extension $Q_1 Q_2$ of $K$ that does not equal either $Q_1$ or $Q_2$.
  \end{enumerate}
\end{definition}

\begin{remark}\label{rem:comp2}
  Suppose that $K$ is of characteristic not equal to $2$, and let $Q_1 = K (\sqrt{d_1})$ and $Q_2 = K (\sqrt{d_2})$. Then the complementary extension to $Q_1$ and $Q_2$ over~$K$ is given by $K (\sqrt{d_1 d_2})$. Similarly, if~$K$ has characteristic $2$, with extensions~$Q_i$ defined by $x^2 + x = a_i$ for $i = 1,2$, then the complementary extension to $Q_1$ and~$Q_2$ over $K$ is given by $x^2 + x = a_1 + a_2$. Definition \ref{def:compext} gives a unified description of these more concrete but separate constructions.
\end{remark}

The following result generalizes \cite[Corollary 3.2]{MWcubic2}.

\begin{corollary}\label{cor:pcvsdisc}
  The purely cubic closure $K'$ is the complementary extension to the resolvent extension $Q$ and $K (\zeta_3)$ over $K$.
\end{corollary}

\begin{proof}
  If $\cha (k) \neq 2$, this is a consequence of Lemma \ref{lem:pcvsdisc}(1), since the quadratic extensions of $K$ defined by $\sqrt{-3}$ and $\zeta_3$ coincide. If $\cha (k) = 2$, this is a consequence of Lemma \ref{lem:pcvsdisc}(2) and Remark~\ref{rem:comp2}.
\end{proof}

Since finite fields admit a unique quadratic extension, we recover the following result as a special case of Corollary \ref{cor:pcvsdisc}. Its first part can also be obtained from \cite[Corollary 4.2]{MWcubic} and \cite[Theorem 5.2]{MWcubic2}, and its second part is also a consequence of \cite[Corollary 3.2]{MWcubic2}.

\begin{corollary}\label{cor:purorgal}
  Suppose that the base field $k = \FF_q$ is finite and that the purely cubic closure $K'$ is a (possibly trivial) constant extension.
  \begin{enumerate}
    \item If $q$ is congruent to $1 \pmod 3$, then $L/K$ is Galois if and only if it is purely cubic.
    \item If $q$ is congruent to $-1 \pmod 3$, then $L/K$ is Galois and impurely cubic if the quadratic resolvent has a root and purely cubic and non-Galois otherwise.
  \end{enumerate}
\end{corollary}

\begin{remark}
  Corollary \ref{cor:pcvsdisc} recovers the classical Kummer-theoretic result that if $k$ contains a third root of unity, a cubic extension $L/K$ is purely cubic if and only if it is Galois. Contrary to the case of finite base fields described in Corollary \ref{cor:purorgal}, it is possible for a cubic extension $L/K$ with a general field of constants $k$ that does not contain a third root of unity to be both non-Galois and impurely cubic. An example of this is the extension of $K = \QQ (x)$ defined by the polynomial
  \begin{equation}
    X^3 - 3 X - \frac{2 (x^2 + 2)}{x^2 - 2} ,
  \end{equation}
  whose purely cubic closure is the constant extension by $\QQ (\sqrt{2})$ and whose resolvent extension is therefore the constant extension by $\QQ (\sqrt{-6})$.
\end{remark}

We need two more notions to simplify and manipulate cubic extensions:

\begin{definition}
  Let a field extension $K'/K$ be given, and let $L'$ be a purely cubic extension of $K'$. We say that the extension $L'/K'$ \emph{descends} if there exists a cubic extension $L$ of $K$ such that $L K'$ is $K'$-isomorphic to $L'$ and $K'$ is the purely cubic closure of $L/K$. In this case $L/K$ is called a \emph{purely cubic descent} (or simply a \emph{descent}) of $L'/K'$.
\end{definition}

\begin{definition}\label{def:twist}
  Let a cubic extension $L_1/ K$ be given. A second cubic extension $L_2 / K$ is called a \emph{twist} (resp.\ \emph{bi-twist}) of $L_1 / K$ if the cubic extensions $L_1 \kbar$ and $L_2 \kbar$ of $K \kbar$ are isomorphic (resp.\ bi-isomorphic).
\end{definition}

One of the main results of \cite{MWcubic3} is a full description of the ramification of a cubic extension. In particular, this allows the determination of the admissible ramification types once the genus of the cubic extension $L/K$ is specified. In what follows we will instead classify cubic extensions with prescribed admissible ramification: Given a prescribed purely cubic closure $K'$ of $K$, we determine whether there exist purely cubic extensions $L'$ of $K'$ that descend to a cubic extension $L$ of $K$ whose ramification at a fixed set of places is prescribed.

\section{Quadratic Extensions of Genus Zero}\label{sec:quad}

Our results extensively use the purely cubic closure of a cubic extension $L/K$ of a function field~$K$, which is a quadratic extension of $K$ by Theorem~\ref{thm:cubclos}. For this reason, and as an illustration of our methods, we prove more well-known results via an alternative inroad, by classifying the possible \textbf{geometric, separable, genus-zero} quadratic extensions $Q$ of $K$. While none of these results are new, we are not aware of a single place where they can all be found together. To keep the article self-contained, and to make the analogies with the cubic case possible as well as to exploit their implicit calculation of certain cohomology groups in later sections, we have included them along with some short proofs.

First we assume that the characteristic of $k$ is unequal to $2$. Let $\Gamma_k = \Gal (\kbar / k)$ denote the absolute Galois group of $k$. Then Kummer theory makes all three steps above explicit, except for the classification of bi-twists:

\begin{proposition}\label{prop:Qkbar}
  Let $Q / K$ be a quadratic extension over a base field $k$ with $\cha (k) \neq 2$.
  \begin{enumerate}
    \item If $k = \kbar$, then $Q/K$ is determined up to isomorphism by its branch locus, which is of total degree two. Moreover, there is a single bi-isomorphism class of such extensions, represented by $Q = K (y)$ with $y^2 = x$.
    \item Given a ramification locus $S$ of total degree two, there exists an extension $Q/K$ with ramification locus $S$, defined by an equation of the form $y^2 = f$ with $f \in K$. Extensions obtained from different ramification loci are pairwise non-isomorphic.
    \item The isomorphism classes of the twists of an extension $y^2 = f$ are in bijective correspondence with the classes $[ c ] \in k^* / (k^*)^2 = H^1 (\Gamma_k, \mmu_2)$. This correspondence sends a class $[ c ]$ to the extension defined by $y^2 = c f$.
  \end{enumerate}
\end{proposition}

To classify bi-twists, we first prove the following proposition, which will give a concrete description of a cohomology group that we will also encounter in Section~\ref{sec:pure}.

\begin{proposition}\label{prop:H1cp}
  Let $(Q_1, D_1)$ and $(Q_2, D_2)$ be two conics with effective divisors of degree two. Then $(Q_1, D_1)$ and $(Q_2, D_2)$ are isomorphic if and only if $Q_1$ and $Q_2$ are isomorphic and the splitting fields of the divisors $D_1$ and $D_2$ coincide.
\end{proposition}

\begin{proof}
  We may reduce to the case $Q_1 = Q_2 = Q$ with non-trivial Galois action. Let $D_1 = a_1 + b_1$ and $D_2 = a_2 + b_2$ over the algebraic closure. Choose a parametrization $\varphi : \PP^1 \to Q$ over the common quadratic extension $q$ of~$k$ with $\varphi (\infty) = a_1, \varphi (0) = b_1, \varphi (1) = a_2$, denoting the coordinate on $\PP^1$ by~$t$. Let $\lambda \in q$ be such that $\varphi (\lambda) = b_2$. Let $\sigma$ be the involution coming from the extension $q$ of $k$. Then $\alpha := \sigma (\varphi)^{-1} \varphi$ sends $\infty, 0, 1$ to $0, \infty, \sigma (\lambda)$ respectively and is therefore the involution $\alpha (t) = \sigma (\lambda) / t$. We get $\sigma (\varphi) = \varphi \alpha$. Now $\sigma (\varphi)$ sends $1$ to $b_2$, whereas $\alpha$ sends $1$ to $\sigma (\lambda)$. We conclude that $\varphi$ maps $\sigma (\lambda)$ to $b_2$, which is also the image of $\lambda$ under $\varphi$. We get that $\sigma (\lambda) = \lambda$ and hence $\sigma (\alpha) = \alpha$.

  Now let $\beta : Q \to Q$ be given by $\beta = \varphi \gamma \varphi^{-1}$, where $\gamma$ is an involution of $\PP^1$ that switches $\left\{ \infty, 0 \right\}$ with $\left\{ 1, \lambda \right\}$. Then $\gamma$ is defined over $k$ and commutes with $\alpha$. Therefore we obtain
\begin{equation}
    \sigma (\beta)
    = \sigma (\varphi) \sigma (\gamma) \sigma (\varphi)^{-1}
    = \varphi \alpha \gamma \alpha^{-1} \varphi^{-1}
    = \varphi \gamma \varphi^{-1} = \beta .
\end{equation}
  The automorphism $\beta : Q \to Q$ is defined over $k$ and sends $D_1$ to $D_2$ by construction.
\end{proof}

\begin{remark}
  The automorphism group of the pair $(\PP^1, \left\{ 0, \infty \right\})$ is $\GGm \rtimes C_2$, with generators given by $x \mapsto \lambda x$ and $x \mapsto 1/x$. Via conjugation, the factor $C_2$ acts on $\GGm$ via inversion. Proposition \ref{prop:H1cp} now furnishes a concrete description of the cohomology group $H^1 (\Gamma_k, \GGm \rtimes C_2)$.
\end{remark}

We can apply the description of the cohomology group thus obtained to our current problem:

\begin{proposition}\label{prop:Qbitwistscharne2}
  The bi-twists of the quadratic extension $y^2 = x$ of $K$ are in bijective correspondence with the isomorphism classes $(Q, D)$ of conics $Q$ with an effective divisor $D$ of total degree two over $k$. Given a pair $(Q, D)$, a corresponding quadratic extension of $K$ is obtained as the quotient $Q \to Q / \iota$. Here $\iota$ is the involution of $Q$ with fixed points $D$.
\end{proposition}

\begin{proof}
  This follows because the bi-automorphism group of the standard extension $y^2 = x$ is isomorphic to $\GGm \rtimes C_2$: the elements $(\lambda, 0)$ act by $(x, y) \mapsto (\lambda^2 x, \lambda y)$, and $(1, 1)$ by $(x, y) \mapsto (1/x, 1/y)$.
\end{proof}

\begin{remark}\label{rem:PP1base}
  Note that the process described in the statement of the proposition always results in a quotient curve isomorphic to $\PP^1$ (cf.\ \cite[Lemma 3.1]{xarles-towers}).
\end{remark}

\begin{remark}
  If $k$ is a finite field, then the bi-isomorphism class of a quadratic extension of $K$ is in fact determined by the Galois structure of its ramification (or branch) locus, as we shall see in Corollary \ref{cor:Qbitwists}. As Proposition \ref{prop:Qbitwistscharne2} and the explicit examples $y^2 = x^2 + 1$ and $-y^2 = x^2 + 1$ over $\RR$ show, this does not remain true over general fields.
\end{remark}

The case where the field $k$ has characteristic $2$ requires different methods. The following result follows from applying the Riemann-Hurwitz formula in its formulation in~\cite[Example 9.4.6]{Vil}:

\begin{proposition}\label{prop:Qbichar2}
  Suppose that $\cha (k) = 2$. Then there is a unique bi-isomorphism class of genus-zero extensions $Q/K$, represented by $Q : y^2 + y = x$.
\end{proposition}

An elementary result for quadratic extensions that will later find cubic analogues in Section \ref{sec:biisotwists}, and which follows from Artin-Schreier theory, is the following:

\begin{proposition}
  Suppose that $\cha (k) = 2$. Then the isomorphism classes of genus-zero quadratic extensions of $K$ that ramify at $\infty$ are in correspondence with the set $k^* \times k/H$, where $H$ is the additive subgroup of $k$ generated by the elements $\gamma^2 + \gamma$ for $\gamma \in k$. To an element $(\alpha, [ \beta ])$ of $k^* \times k/H$ is associated the isomorphism class of the extension defined by
  \begin{equation*}
    y^2 + y = \alpha x + \beta .
  \end{equation*}
  In particular, when $k$ is finite, the number of these isomorphism classes equals $2\cdot \vert k^\times \vert$. Two such classes are twists of one another if and only if their first entries coincide. Therefore the set of twists of a given isomorphism class is parametrized by $k/H$, and this set is of cardinality two if $k$ is finite.
\end{proposition}

Summarizing this section, the next corollary follows from Proposition~\ref{prop:Qbitwistscharne2} and Proposition~\ref{prop:Qbichar2}:

\begin{corollary}\label{cor:Qbitwists}
  Let $k$ be a finite field. If $\cha (k) \ne 2$, then there are exactly two bi-isomorphism classes of genus-zero extensions $Q/K$, given by $Q_1 : y^2 = x$ and $Q_2 : y^2 = f$ respectively. Here $f$ is any irreducible polynomial of degree two over $k$.

  If $\cha (k) = 2$, then there is a single bi-isomorphism class, represented by $Q : y^2 + y = x$.
\end{corollary}

\begin{remark}\label{rem:ratptinf}
  If $k$ is a finite field with $\cha (k) \ne 2$, then we can normalize $Q_2$ further by taking $f = x^2 - d$ with $d$ a preferred non-square in $k$. Moreover, the conic $Q_1$ (resp.\ $Q_2$, resp.\ $Q$) above then admits the rational point $(1 : 0 : 0)$ (resp.\ $(1 : 1 : 0)$, resp.\ $(1 : 0 : 0)$) above infinity.
\end{remark}

\section{The Purely Cubic Case}\label{sec:pure}

In this section we describe purely cubic extensions $L/K$ of a rational function field $K = k (x)$. We also classify their twists, as well as their bi-twists for purely cubic extensions of genus zero or one.

Assume that the characteristic of $k$ is unqual to $3$ and let $L/K$ be a cubic extension with generator~$y$ whose minimal polynomial is $X^3 - \beta \in K[X]$ for some $\beta \in K$. By Kummer theory, the ramified places of $L/K$ are precisely those places (finite or infinite) such that $(v_{\pp}(\beta),3 )= 1$. Writing $\beta = \prod_{i=1}^s \beta_i^{e_i}$, where $e_i=1,2$ and the $\beta_i$ are distinct irreducible polynomials of respective degrees $d_i$, we see that the finite places $\pp$ of  $K$ such that $( v_\pp(\beta),3) = 1$ are precisely the zero divisors $\pp_i$ of~$\beta_i$ in $K$. Moreover, $( v_{\infty}(\beta),3)=1$ if and only if $(d,3)=1$, where $d = \sum_{i=1}^s e_i d_i$ is the degree of $\beta$. (Note that $\deg( \pp_i) =~d_i$ and $\deg(\infty) = 1$.)

\begin{theorem}\label{thm:purelycubic}
Let $T$ be a set of places of $K$ of cardinality $t$.
\begin{enumerate}
\item Suppose that $\infty \not\in T$ and that the places of $T$ correspond to monic polynomials $P_1, \cdots, P_{t}$ which are indexed such that
  \begin{equation}\label{eq:orderPi}
    \deg(P_i) \equiv
    \begin{cases}
      0 \pmod{3} & \text{ for } 1 \leq i \leq s, \\
      -1 \pmod{3} & \text{ for } s+1 \leq i \leq r, \\
      1 \pmod{3} & \text{ for } r+1 \leq i \leq t.
    \end{cases}
  \end{equation}
Then any purely cubic extension $L/K$ which is fully ramified precisely over the places of $T$ admits a generator $y$ satisfying the equation
\begin{equation}
      y^3 = u P_1^{\epsilon_1} P_2^{\epsilon_2} \cdot \ldots \cdot  P_s^{\epsilon_s}\frac{P_{r+1}^{\epsilon_{r+1}} \cdot \ldots \cdots P_t^{\epsilon_t}}{ P_{s+1}^{\epsilon_{s+1}} \cdot \ldots \cdots P_r^{\epsilon_r}} ,
\end{equation}
where $\epsilon_{1}=1$, and $\epsilon_i \in \{ \pm 1 \} \ (i = 2, \ldots, t)$ satisfy $ \sum_{i=s+1}^t \epsilon_i \equiv 0 \pmod 3$, and where $u \in k$ is a unit.
\item
Suppose that $\infty \in T$ and that the $t-1$ finite places in $T$ correspond to monic polynomials $P_1, \ldots, P_{t-1}$ indexed as in \eqref{eq:orderPi}.
Then any purely cubic extension $L/K$ which is fully ramified precisely over the places of $T$ admits a generator $y$ satisfying the equation
\begin{equation}
      y^3 = u P_1^{\epsilon_1} P_2^{\epsilon_2} \cdot \ldots \cdot  P_s^{\epsilon_s}\frac{P_{r+1}^{\epsilon_{r+1}} \cdot \ldots \cdots P_{t-1}^{\epsilon_{t-1}}}{ P_{s+1}^{\epsilon_{s+1}} \cdot \ldots \cdots P_r^{\epsilon_r}} ,
\end{equation}
    where $\epsilon_{1}=1$, $\epsilon_i \in \{ \pm 1 \} \ (i = 2, \ldots, t)$ satisfy $ \sum_{i=s+1}^t \epsilon_i \not\equiv 0 \pmod 3$, and where $u \in k$ is a unit.
  \end{enumerate}
  In both cases, there are $(1/3) 2^s (2^{t - s - 1} - (-1)^{t - s - 1})$ such covers up to isomorphism over $\kbar$.
\end{theorem}

\begin{proof}
  Index the fully ramified (possibly infinite) places $\mathcal{P}_1, \ldots, \mathcal{P}_t$ of $L/K$ such that
\begin{equation}
    \deg(\mathcal{P}_i) \equiv
    \begin{cases}
      0 \pmod{3} & \text{ for } 1 \leq i \leq s, \\
      -1 \pmod{3} & \text{ for } s+1 \leq i \leq r, \\
      1 \pmod{3} & \text{ for } r+1 \leq i \leq t.
    \end{cases}
\end{equation}
  Then there exists a Kummer generator $z$ for $L/K$ such that
\begin{equation}
    \mathrm{div}(z^3) =  \frac{ \prod_{i=1}^{r} \mathcal{P}_i^{\epsilon_i}}{ \prod_{i= r+1}^t \mathcal{P}_i^{\epsilon_i}}
\end{equation}
  such that $\sum_{i=s+1}^t \epsilon_i \equiv 0 \pmod 3$.
  (Another valid Kummer generator is obtained by mapping $\epsilon_i \mapsto -\epsilon_i$ for all $i = 1, \ldots, t$.)

  When $t=s$, the statement of the theorem follows by identifying $P_i \mapsto \mathcal{P}_i$, which determines the minimal polynomial of the generator up to a unit $u$.

  So from now on we assume that $t-s \geq 1$ and set up a linear recursion: Let $E_k$ be the number of tuples $(\epsilon_1, \ldots, \epsilon_k)$ with $\epsilon_i = 1, -1 \pmod 3$ such that $\sum_i \epsilon_i \equiv 0 \pmod 3$, and $F_k$ the number of tuples with $\sum_i \epsilon_i \equiv 1, -1 \pmod 3$. Then $E_1 = 0$ and $F_1 = 2$, and
\begin{equation}
    E_{i+1} = F_i \text{ and } F_{i+1} = 2 E_i + F_i
\end{equation}
  for $i \geq 1$.
  This gives a matrix $\left( \begin{smallmatrix}0 & 1 \\ 2 & 1 \end{smallmatrix} \right)$ with eigenvalues $2$ and $-1$; we may write
\begin{equation}
    \left( \begin{smallmatrix}0 & 1 \\ 2 & 1 \end{smallmatrix} \right) = -\frac{1}{3} \left( \begin{smallmatrix}-1 & 1 \\ 1 & 2 \end{smallmatrix} \right) \left( \begin{smallmatrix}-1 & 0 \\ 0 & 2 \end{smallmatrix} \right) \left( \begin{smallmatrix} 2 & -1 \\ -1 & -1 \end{smallmatrix} \right).
\end{equation}
  Then our recurrence is solved by
\begin{equation}
    \left( \begin{smallmatrix} E_k \\ F_k \end{smallmatrix} \right) = -\frac{1}{3} \left( \begin{smallmatrix}-1 & 1 \\ 1 & 2 \end{smallmatrix} \right) \left( \begin{smallmatrix}(-1)^{k-1} & 0 \\ 0 & 2^{k-1} \end{smallmatrix} \right) \left( \begin{smallmatrix} 2 & -1 \\ -1 & -1 \end{smallmatrix} \right)\left( \begin{smallmatrix} E_1 \\ F_1 \end{smallmatrix} \right) = \frac{2}{3}\left( \begin{smallmatrix} 2^{k-1} - (-1)^{k-1} \\ 2^k - (-1)^k \end{smallmatrix} \right),
\end{equation}
  showing in particular that there exist solutions for any $k$.
  Again identifying $P_i \mapsto \mathcal{P}_i$, we see that now $\epsilon_i \in \{1,-1\}$ for $1 \leq i \leq s$ may be chosen arbitrarily, while the above recursion confirms that there exist compatible choices of $\epsilon_i$ for $s+1 \leq i \leq t$, which determine the extension up to a unit $u$ as before.
  \end{proof}

\begin{remark}\label{rem:conic1}
  Using the theory of divisors, one can show that the number of covers up to twists is unchanged when $K$ is the function field of a non-trivial conic over $k$ instead of a rational function field over $k$. We have restricted ourselves to $K = k (x)$ for simplicity. For the function field $k (C)$ of a general conic $C$ over $k$, we proceed as in Case (1) of Theorem \ref{thm:purelycubic}, replacing $P_i$ with places $D_i$: under the given hypotheses on the $\epsilon_i$ there then exists a function $f \in k (C)$ whose divisor of poles and zeros contains exactly the $D_i$, and we can take equations $y^3 = u f$ for the corresponding cubic extensions.
\end{remark}

As in the case of Proposition~\ref{prop:Qkbar}, one shows the following proposition using Kummer theory.

\begin{proposition}\label{prop:pctwists}
  The isomorphism classes of the twists of an extension $y^3 = f$ are in bijective correspondence with the classes $[ c ] \in k^* / (k^*)^3 = H^1 (\Gamma_k, \mmu_3)$. This correspondence sends a class $[ c ]$ to the extension defined by $y^3 = c f$.
\end{proposition}

We describe bi-isomorphism classes of purely cubic extensions in two cases, namely those where the total ramification locus has degree either two or three. In the former case the extension is isomorphic over $\kbar$ to the one defined by $y^3 = x$.

\begin{theorem}\label{thm:biisog0}
  The bi-twists of the extension $y^3 = x$ of $K$ are in bijective correspondence with the isomorphism classes $(Q, D)$ of conics $Q$ with an effective divisor $D$ of total degree two over $k$. Given a pair~$(Q, D)$, the corresponding twist can be obtained as follows:
  \begin{enumerate}
    \item Let $q$ be the smallest extension of $k$ splitting $D$, and let $\iota : \PP^1 \to Q$ be a morphism over $q$ that sends $\left\{ 0, \infty \right\}$ to $D$.
    \item We have $\sigma (\iota) = \iota \alpha$, where $\alpha (t) = \lambda / t$ for the affine coordinate $x$ of $\PP^1$ with $\lambda \in k$. Let $\phi_0 : \PP^1 \to \PP^1$ be the map sending $x \mapsto x^3 / \lambda$.
    \item Let $\phi = \iota \phi_0 \iota^{-1}$. Then $\phi : (Q, D) \to (Q, D)$ is a twist of $\phi_0$ that ramifies exactly in $D$.
  \end{enumerate}
  There is only one bi-twist that is purely cubic, which is the trivial one. The bi-twists that are Galois are those for which the splitting field of $D$ equals $k (\zeta_3)$.
\end{theorem}

\begin{proof}
  As in Proposition~\ref{prop:Qbitwistscharne2}, the first description follows because the bi-automorphism group of the standard extension $y^3 = x$ is again isomorphic to $\GGm \rtimes C_2$. The second statement follows from \cite[Theorem 3.19]{LRS-Explicit}. Finally, a twist is purely cubic if and only if its branch locus has trivial Galois action, which corresponds to the trivial case where $(Q, D)$ is isomorphic to $(\PP^1, \left\{ 0, \infty \right\})$. In light of Corollary~\ref{cor:pcvsdisc} this implies that a twist is Galois if and only if the splitting field of its branch locus equals $k (\zeta_3)$.
\end{proof}

\begin{corollary}\label{cor:biisog0}
  If the base field $k = \FF_q$ is finite, then the extension $y^3 = x$ has exactly two bi-twists. Exactly one of these twists is Galois, namely the trivial twist if $q \equiv 1 \pmod 3$ and the non-trivial twist if $q \not\equiv 1 \pmod 3$.
\end{corollary}

\begin{proof}
  The first statement follows from Theorem \ref{thm:biisog0} because any conic over $k$ is isomorphic to a projective line, combined with the fact that $k$ admits a unique non-trivial quadratic extension. The second follows by using Corollary \ref{cor:purorgal}.
\end{proof}

\begin{remark}
  In Proposition \ref{prop:33} we will find an explicit defining equation for the non-trivial impurely cubic twist in Corollary \ref{cor:biisog0} by using the methods from Section~\ref{sec:impure}.

  When $Q \not\cong \PP^1$, we obtain bi-twists for which the base curve is no longer a projective line. Explicit equations can then be obtained by using \cite[Proposition~3.16]{LRS-Explicit}.
\end{remark}

Now consider the second case, for which the total ramification locus has degree three. In this case, the extension is isomorphic over $\kbar$ to that defined by $y^3 = x (x - 1)$. We again determine the corresponding bi-twists.

\begin{proposition}\label{prop:333biaut}
  The bi-automorphism group of $y^3 = x (x - 1)$ is isomorphic to $\mu_3 \times S_3$. Here $\mu_3$ acts via $(x, y) \mapsto (x, \zeta_3 y)$, whereas the generators $(0 \, \infty)$ and $(1 \, \infty)$ of $S_3$ act via $(x, y) \mapsto (1/x, -y/x)$ and $(x, y) \mapsto (x / (x - 1), y / (x - 1)$, respectively.
\end{proposition}

\begin{proof}
   An explicit calculation shows that the claimed maps are bi-automorphisms and generate the indicated group. That these generate the full bi-automorphism group follows from the fact that any bi-automorphism must map the branch locus $\left\{ 0, 1, \infty \right\}$ of $(x, y) \mapsto x$ to itself, and that a bi-automorphism that fixes $x$ necessarily sends $y$ to $\zeta_3^i y$ for some $i$.
\end{proof}

\begin{theorem}\label{thm:biisog1}
  The bi-twists of the extension $y^3 = x (x - 1)$ of $K$ are in bijective correspondence with the pairs $(L, [ c ])$, where $L$ is an isomorphism class of an extension of degree at most three of $K$ and where $[ c ] \in k^* / (k^*)^3$. Given a pair $(L, [ c ])$, a corresponding twist is given by:
  \begin{enumerate}
    \item $y^3 = c x (x - 1)$ if $L = K$.
    \item $y^3 = c f$, where $f$ is the minimal polynomial of a generator of $L/K$, if $L \neq K$.
  \end{enumerate}
  In particular, all these twists up bi-isomorphism are purely cubic. They are Galois if and only if $k$ contains $\zeta_3$.
\end{theorem}

\begin{proof}
  This result follows by considering the twists corresponding to the factors of $\mu_3 \times S_3$ (cf. Proposition \ref{prop:333biaut}) separately. The factor $\mu_3$ gives rise to multiplication of the polynomial defining the branch locus in $x$ by a scalar $c$. On the other hand, via the indicated automorphisms the cohomology set $H^1 (\Gamma_k, S_3)$ classifies the Galois stable subsets of $\PP^1_{\kbar}$ of degree three that classify the branch locus. By transitivity, two such subsets of $\PP^1_{\kbar}$ are isomorphic if and only if they are isomorphic as $\Gamma_k$-sets, which is the case if and only if they have isomorphic stabilizers under the Galois action. The Galois correspondence between (sub)extensions and stabilizers then yields the result.
\end{proof}

\begin{remark}
  A given cubic extension $y^3 = g$ of $K$ can be transformed explicitly into one of the forms of Theorem \ref{thm:biisog1} as follows. Let $f$ be the polynomial whose splitting field is isomorphic to that of $g$. Label the roots $\alpha_1, \alpha_2, \alpha_3$ of $f$ and $\beta_1, \beta_2, \beta_3$ of $g$ such that the Galois action on them coincides. Then the Möbius transformation mapping $\alpha_i \mapsto \beta_i$ is defined over $k$. Pulling back $g$ by it and absorbing third powers in the denominator if necessary gives a defining equation of the form $y^3 = d f$, which after a transformation in $y$ gives rise to a standard equation $y^3 = c f$ for which the classes of $c$ and $d$ in $k^* / (k^*)^3$ coincide.
\end{remark}

Since a finite field has exactly three isomorphism classes of extensions of degree up to three, we obtain the following:

\begin{corollary}
  If the base field $k = \FF_q$ is finite, then the extension $y^3 = x (x - 1)$ has exactly nine bi-twists if $q \equiv 1 \pmod 3$, and three bi-twists if $q \equiv -1 \pmod 3$.
\end{corollary}

\section{Descent from the Purely Cubic Closure}\label{sec:impure}

Let sets of places $S$ and $T$ of $K$ be given, together with a quadratic extension $K'$ of $K$ of genus zero. Write $t = |T|$. The following two results give a necessary and sufficient condition for the existence of a cubic extension of~$K$ that ramifies partially (resp.\ fully) at the places in $S$ (resp.\ $T$) and has purely cubic closure~$K'$.

\begin{theorem}\label{thm:norms}
  There exists a cubic extension $L$ of $K$ such that $K'$ is the purely cubic closure of $L/K$ and with $L$ fully ramified at exactly the places in $T$ if and only if all places in $T$ split in the quadratic genus-zero extension $K'/K$.
\end{theorem}

\begin{proof}
  Suppose that $L/K$ is as described, and use the description of $K'$ as $K (s)$ in Theorem \ref{thm:cubclos}. Then as $L K(s)/K$ has degree six, the fully ramified places in $L/K$ are the places below the fully ramified places of $L(s)/K(s)$. By Kummer theory, the ramified places of $L(s)/K(s)$ are at the zeros and poles of $s$ whose valuation is not a multiple of three. Let $\pP$ be such a place. Then we have
\begin{equation}
    0 \neq
      v_{\sigma (\pP)} (s)
    = v_{\pP} (\sigma (s))
    = v_{\pP} (1 / s)
    = -v_{\pP} (s).
\end{equation}
  This implies in particular that $\sigma (\pP) \neq \pP$, so that $\pP$ is split. As in Theorem~\ref{thm:cubclos}, we see that $L$ is the fixed field of $L(s)$ under the involution of $L (s)$ over $K$ which sends $(g, w)$ to $(g^{-1}, w^{-1})$.

  In order to show sufficiency, let $\pP_1, \dots, \pP_{t}$ be places in $K'$ above the places $\pp_1, \dots , \pp_{t}$ in $T$. Consider the divisor
  \begin{equation}\label{eq:signchoice}
    D = \sigma (\pP_1) - \pP_1
      + \dots
      + \sigma (\pP_{t}) - \pP_{t} .
  \end{equation}
  Since $K'$ is of genus zero, the Riemann--Roch theorem implies that there is a function $f$ such that $D = \mathrm{div}(f)$. We will consider a cubic extension $L'$ of $K'$ defined by $w^3 = g = \lambda f$ with $\lambda \in k$, for which we will show that there exists a cubic extension~$L$ of $K$ such that $L K' = L'$. In this case $K'$ is by construction the purely cubic closure of $L/K$, and the ramification of $L$ over $K$ is of the requested type since all triple ramification in $L'/K'$ comes from the subextension $L/K$.

  Theorem \ref{thm:cubclos} shows that to descend $L'$ it is sufficient and necessary to prolong the involution $\sigma$ of $K'$ to it. We have $\sigma (D) = -D$ by construction, which implies that $\sigma (f) = \lambda / f$ for some $\lambda$. Note that $\lambda$ is in $k$. This follows because
  \begin{equation}
    f = \sigma (\sigma (f)) = \sigma (\lambda) / \sigma (f)
      = \sigma (\lambda) / (\lambda / f)
      = \frac{\sigma (\lambda)}{\lambda} f
  \end{equation}
  so that indeed $\sigma (\lambda) = \lambda$. Now let $g = \lambda f$. Then
  \begin{equation}
    \sigma (g) = \sigma (\lambda) \sigma (f)
    = \lambda \cdot \lambda / f = \lambda^3 / g.
  \end{equation}
  For the extension $L'$ defined by $w^3 = g$, we can prolong $\sigma$ by setting $\sigma (w) = \lambda / w$, since then we have the compatibility
  \begin{equation}
    \sigma (w^3) = \lambda^3 / w^3 = \lambda^3 / g = \sigma (g) .
  \end{equation}
  This proves our existence claim for $L'$ and furnishes the requested cubic extension~$L/K$.
\end{proof}

The proof of Theorem \ref{thm:norms} also shows the following.

\begin{corollary}\label{cor:allwithclos}
  Up to $\kbar$-isomorphism there are either $0$ or $2^{t - 1}$ extensions $L/K$ as in Theorem~\ref{thm:norms}.  
\end{corollary}

\begin{proof}
  Kummer theory shows that over $\kbar$ a cubic extension is determined by its signed ramification locus as in \eqref{eq:signchoice} modulo the global negation corresponding to the isomorphism sending $y$ to $1/y$. Since there are $t$ sign choices in \eqref{eq:signchoice}, we obtain the result.
\end{proof}

\begin{remark}\label{rem:serre}
  For the base field $k = \CC$, the number of covers up to isomorphism can also be determined using a formula due to Serre \cite{Serre-Topics}. It states that if $G$ is a finite group and $C_1, \dots, C_n$ are conjugacy classes in $G$, then the number of solutions of the equation
\begin{equation}\label{eq:gi}
    g_1 \dots g_n = 1, \qquad g_i \in C_i,
\end{equation}
  equals
  \begin{equation}\label{eq:serreeq}
    \frac{1}{|G|} |C_1| \dots |C_n| \sum_{\chi} \frac{\chi (C_1) \dots \chi (C_n)}{\chi (1)^{n - 2}} ,
  \end{equation}
  where $\chi$ runs over the irreducible characters of $G$.

  Using covering theory, the isomorphism classes of the covers in Corollary \ref{cor:allwithclos} correspond to the solutions up to conjugacy of \eqref{eq:serreeq}, for $C_1, \dots C_2$ corresponding to $2$-cycles and $C_3, \dots, C_{t + 2}$ corresponding to $3$-cycles. Using the character table of $S_3$, the number of such solutions equals
\begin{equation}
    (1/6) 3^2 2^{t} (1 + 1 + 0) = 3 \cdot 2^{t} .
\end{equation}
  Since $t > 0$, all these solutions generate $S_3$ and therefore do not admit automorphisms, implying that their total number is $(1/6) (3 \cdot 2^{t}) = 2^{t - 1}$. In fact, for arbitrary branch loci $S$ and $T$ with $s = |S| \geq 2$ and $t = |T| \geq 1$ we get $3^{s - 2} 2^{t - 1}$ covers up to isomorphism as long as $s$ is even. For parity reasons, the number of solutions of \eqref{eq:gi} is zero when $s$ is odd. When $s = 0$, we are in the situation of the previous section. In this case a slight modification of the argument, using $G = \ZZ / 3 \ZZ$ instead, indeed yields the number of covers in Theorem \ref{thm:purelycubic}. In this sense, the results in that theorem and Corollary \ref{cor:allwithclos} are more explicit versions of these abstract calculations. Note the independence of these numbers of $\cha (k)$ when this characteristic does not equal $2$ or $3$, in line with \cite[Expos{\'e} 13, Corollaire 2.12]{SGA1}.
\end{remark}

\begin{remark}
  The previous counting result over $\CC$ extends without trouble to the case where the purely cubic closure is of genus larger than zero. However, the methods of Theorem \ref{thm:norms} do not apply, since we cannot conclude that the relevant divisor $D$ is principal. Phrased differently, the issue is that the affine coordinate ring of the projective curve corresponding to $K$ minus a point is only a unique factorization domain if $K$ is of genus $0$.

  When considering curves of higher genus, being able to construct $f$ requires the $3$-divisibility of $D$ in the Jacobian of $K'$. This divisibility is not guaranteed when the base field is not algebraically closed. Moreover, the isomorphism class of the resulting extension depends on the choice of a quotient. We will also encounter this phenomenon in Section \ref{sec:parshin}.
\end{remark}

\begin{remark}\label{rem:startwithS}
  For fixed specified branch loci $S$ and $T$, there may exist multiple genus-zero extensions~$K'$ of $K$ in which $T$ splits. To see this, we first consider the case where $S = \emptyset$ and where $T$ is the single degree-four place of $K = k (x)$ defined by the Galois orbit of $1 + \sqrt{2} + \sqrt{3}$ over $k = \QQ$. Then $T$ splits over the different constant quadratic extensions $K (\sqrt{2})$, $K (\sqrt{3})$, and $K (\sqrt{6})$, as do the pairs of places defined by the pairs of polynomials $\left\{ x^2 + (\pm \sqrt{2} - 2) x \mp \sqrt{2} \right\}$, $\left\{ x^2 + (\pm \sqrt{3} - 2) x + (\mp \sqrt{3} + 2) \right\}$, and $\left\{ x^2 - 2 x + (\pm \sqrt{6} - 4) \right\}$ respectively.

  In order to obtain an example with non-constant purely cubic closure, consider the case where $S = \left\{ \infty, 0 \right\}$ and where $T$ is the degree-two place defined by the Galois orbit of $6 + 4 \sqrt{2}$. Then $T$ splits over both the non-isomorphic quadratic extensions $K (\sqrt{x})$ and $K (\sqrt{2 x})$ that ramify over $S$. For the former extension, described by $y^2 = x$, we obtain the pairs of places described by the Galois orbits of points $(x, y)$ given by $\left\{ (6 \pm 4 \sqrt{2}, 2 \pm \sqrt{2}) \right\}$ and $\left\{ (6 \pm 4 \sqrt{2}, -(2 \pm \sqrt{2})) \right\}$. For the latter extension, described by $y^2 = 2 x$, we obtain the orbits given by $\left\{ (6 \pm 4 \sqrt{2}, 1 \pm \sqrt{2}) \right\}$ and $\left\{ (6 \pm 4 \sqrt{2}, -(1 \pm \sqrt{2})) \right\}$.

  In light of Proposition \ref{prop:Qkbar}.(2), the purely cubic closures $K'$ that give rise to covers $L / K$ are quadratic twists of one another. Lemma \ref{lem:pcinv} will show that the purely cubic closure is invariant under taking twists. In particular, cubic extensions~$L / K$ obtained by using different $K'$ are never isomorphic.
\end{remark}

Let $S$, $T$, $K'$ be fixed as above, and let $L/K$ be a corresponding extension. We determine the twists of $L/K$ in the sense of Definition \ref{def:twist}. In the purely cubic case the twists are described by Kummer theory, as in Proposition \ref{prop:pctwists}.

\begin{proposition}\label{prop:notwists}
  Let $L$ be a cubic extension of $K$ with non-constant purely cubic closure, or equivalently a cubic extension of $K$ whose set $S$ of partially ramified places is non-empty. Then $L/K$ has trivial automorphism group over $\kbar$ and therefore does not admit non-trivial twists.
\end{proposition}

\begin{proof}
  The group $\Aut (L \kbar / K \kbar)$ is trivial, since otherwise $L \kbar$ would be Galois over $K \kbar$ and $S$ would be empty. This implies that the cohomology group $H^1 (\Gamma_k, \Aut (L \kbar / K \kbar))$ that parametrizes the twists consists of a single element.
\end{proof}

If the purely cubic closure is constant, then we will classify the twists using a cohomology group. Before embarking on this, we note the following:

\begin{lemma}\label{lem:pcinv}
  The purely cubic closure is invariant under taking twists in the following sense: if $L'/K$ is a twist of $L/K$, then the purely cubic closures of these extensions coincide.
\end{lemma}

\begin{proof}
  We may assume that the purely cubic closure is constant in light of Proposition \ref{prop:notwists}. Moreover, Proposition \ref{prop:pctwists} shows that the twist of a purely cubic extension is again purely cubic. Now let~$m$ be an extension of the field of constants $k$ for which the base extension $L m / K m$ is purely cubic. Then $L' m / K m$ is a twist of $L m / K m$ and hence also purely cubic. Swapping the roles of $L$ and $L'$ shows that the extensions $m$ for which $L m / K m$ is purely cubic are exactly those for which $L' m / K m$ is purely cubic. In particular, the smallest such field, that is, the purely cubic closure, coincides for both extensions.
\end{proof}

\begin{proposition}
  Let $L/K$ be a cubic extension with constant purely cubic closure~$K'$. Then by Corollary \ref{cor:pcvsdisc} the resolvent extension $Q$ of $L/K$ is also constant. Write $Q = q K$, where $q$ is an extension of $k$. Then the twists of $L/K$ are in bijective correspondence with the cohomology group $H^1 (\Gamma_k, (\ZZ / 3 \ZZ)_q)$, where $(\ZZ / 3 \ZZ)_q$ is the group scheme of order three over $k$ on whose non-trivial elements the Galois action is defined by the quadratic extension $q$.
\end{proposition}

\begin{proof}
  The first claim follows from Corollary \ref{cor:pcvsdisc}. It implies that the group $\Aut (L \kbar / K \kbar)$ is a Galois module of cardinality $3$. The non-trivial elements of this group cut out the quadratic extension $q$ of~$k$, which implies the second claim.
\end{proof}

For extensions with non-trivial constant purely cubic closure, one can describe the twists via a cohomology group. This description will be a consequence of the following more generally useful lemma:

\begin{lemma}\label{lem:pcisodesc}
  Let $L_1$ and $L_2$ be two cubic extensions of $K$. Then $L_1$ and $L_2$ are isomorphic if and only if they have the same purely cubic closure $K'$ and the extensions $L_1 K'$ and $L_2 K'$ of $K'$ are isomorphic.
\end{lemma}

\begin{proof}
  One direction is clear. Conversely, let $L_1$ and $L_2$ be cubic extensions whose purely cubic closures~$K'$ coincide, and suppose that $L_1 K'$ and $L_2 K'$ are isomorphic over $K'$. Let $M$ be a field isomorphic to both, and embed $L_1$ and $L_2$ into it. We know that $L_1$ and $L_2$ are both fixed fields under involutions of $M$ that prolong the involution of $K'$ over $K$.

  If $\Aut (M / K')$ is trivial, then the involutions coincide, and so $L_1$ and $L_2$ will coincide as subfields of~$M$, so that in particular they will be isomorphic. If $\Aut (M / K')$ is non-trivial, then $K'$ is the common resolvent extension of $L_1$ and $L_2$. This implies that $M$ is the Galois closure of both $L_1$ and $L_2$. But by Galois theory the cubic subfields of $M$ are all isomorphic.
\end{proof}

Let $q$ be a non-trivial quadratic extension of the base field. Let $L/K$ be a cubic extension with purely cubic closure $K' = q K$, and let $L' = L K'$. Choose a defining equation $y^3 = f$ for $L'/K'$.

\begin{lemma}
  Let $c \in q$, and let $M'$ be the extension of $K'$ defined by $y^3 = c f$. Then $M'$ descends if and only if the norm $\Nm (c) = c \sigma (c)$ of $c$ from $q$ to $k$ is a third power in~$k$.

  This construction yields a bijective correspondence between the set of twists of~$L/K$ and the kernel of the induced map
  \begin{equation}\label{eq:normcubes}
    \begin{split}
      q^* / (q^*)^3 & \to k^* / (k^*)^3 \\
      [ c ] & \mapsto [ \Nm (c) ] .
    \end{split}
  \end{equation}
  To the class of $c$ in this kernel this correspondence associates the descent of $y^3 = c f$ (which is uniquely determined by Lemma \ref{lem:pcisodesc}).
\end{lemma}

\begin{proof}
  As in the proof of Theorem \ref{thm:norms}, we have to construct a suitable $K$-involution of $M'$. In light of Theorem \ref{thm:cubclos}, we may assume that the original involution sends $f$ to $1/f$ and $y$ to $1/y$. The involution for $y^3 = c f$ should then still send $y$ to $\lambda / y$ for some $\lambda \in K'$. This defines an actual involution if and only if
  \begin{equation}
    y = \sigma (\sigma (y)) = \sigma (\lambda / y) = \sigma (\lambda) / \sigma (y) = \frac{\sigma (\lambda)}{\lambda} y,
  \end{equation}
  which means that $\lambda \in k$. Moreover, we need that the involution respects the defining equation, so
  \begin{equation}
    \frac{\lambda^3}{y^3} = \sigma (y^3) = \sigma (c f) = \sigma (c) \sigma (f) = \frac{\sigma (c)}{f} = \frac{c \sigma (c)}{c f} = \frac{c \sigma (c)}{y^3} .
  \end{equation}
  This gives the criterion for $M'$ to descend as given in the statement of the lemma. The rest of its statement then follows from Proposition \ref{prop:pctwists} and Lemma \ref{lem:pcisodesc}.
\end{proof}

\begin{corollary}
  The cohomology group $H^1 (\Gamma_k, (\ZZ / 3 \ZZ)_q)$ is isomorphic to the kernel of the homomorphism $q^* / (q^*)^3 \to k^* / (k^*)^3$ induced by the norm map as in Eq.~\eqref{eq:normcubes}.
\end{corollary}

\begin{lemma}\label{lem:twistisos}
  There are canonical isomorphisms
  \begin{equation}
    H^1 (\Gamma_k, (\ZZ / 3 \ZZ)_q) \cong \ker (q^* / (q^*)^3 \to k^* / (k^*)^3) \cong N_1 / N_1^3,
  \end{equation}
  where $N_1$ is the group of elements in $q^*$ with norm $1$ (which by Hilbert's Theorem~90 is isomorphic to $(q^* / k^*)$). As a result, the twists of the unique descent of $y^3 = f$ are parametrized by the unique descents of the extensions $y^3 = c f$, where $c$ runs though representatives of the quotient $N_1 / N_1^3$.
\end{lemma}

\begin{proof}
  We have already proved the existence of the first isomorphism. To obtain the second, we apply the Snake Lemma to the diagram
\begin{equation}
    \begin{tikzcd}
      1 \ar[d] \ar[r] & (q^*)^3 \ar[d] \ar[r] & q^* \ar[d] \ar[r] & q^* / (q^*)^3 \ar[d] \ar[r] & 1 \ar[d] \\
      1 \ar[r] & (k^*)^3 \ar[r] & k^* \ar[r] & k^*/(k^*)^3 \ar[r] & 1
    \end{tikzcd}
\end{equation}
  It then suffices to show the following:
  \begin{enumerate}
    \item The kernel of the norm homomorphism $(q^*)^3 \to (k^*)^3$ is $N_1^3$;
    \item The canonical map $(k^*)^3 / \Nm ((q^*)^3) \to k^* / \Nm (q^*)$ has trivial kernel.
  \end{enumerate}

  For (1), let $\beta \in q^*$ with $\beta = \delta^3$ with $\delta \in q^*$ be in the kernel. Then $\Nm (\delta^3) = \Nm (\beta) = 1$. If~$k$ does not contain $\zeta_3$, then this immediately implies $\Nm (\delta) = 1$ and we are done. If $k$ does contain~$\zeta_3$, and $\Nm (\delta) \neq 1$, then we may assume that $\Nm (\delta) = \zeta_3$. But then also $\beta = (\zeta_3 \delta)^3$, where now $\Nm (\zeta_3 \delta) = \zeta_3^2 \zeta_3 = 1$, so that once again we find $\beta \in N_1^3$.

  For (2), suppose that $\alpha \in k^*$ with $\alpha = \gamma^3$ with $\gamma \in k^*$ is such that $\alpha \in \Nm (q^*)$. We want to show that $\gamma$ is a norm as well. This follows from the fact that the group $(k^*) / \Nm (q^*)$ is $2$-torsion, so that the class of $\gamma$ in it coincides with that of $\alpha$.
\end{proof}

It remains to make the results so far explicit in terms of defining equations. Recall that by the arguments in \cite{MWcubic2}, cf.~Theorem~\ref{thm:cubclos}, we already have the following general existence result.

\begin{theorem}\label{thm:defeq}
  Let $L/K$ be an impurely cubic extension whose purely cubic closure $K'= K (r)$ is defined an element $r$ with minimal polynomial $X^2 - aX - b$ over $K$. Let~$\sigma$ generate the Galois group of $K'$ over $K$. Then there exist elements $P, Q \in~K$ and an element $\theta = P + Q r \in K$ such that $L/K$ admits a defining equation $x^3-3x-\alpha$ with
  \begin{equation}\label{eq:alpha}
    \alpha = \frac{\sigma (\theta)}{\theta} + \frac{\theta}{\sigma (\theta)}
    =  \frac{P + Q \sigma(r)}{P + Q r} + \frac{P + Q r}{P + Q \sigma(r)}
    = \frac{2(P^2+aPQ)+(a^2+2b)Q^2}{P^2+aPQ - bQ^2}.
  \end{equation}
\end{theorem}

We now restrict to the situation where sets of places $S$ and $T$ of partial and total ramification are specified, as well as a corresponding purely cubic closure $K'$ of genus zero whose ramifying places coincide with those in $S$ and that splits the places in $T$. We aim to obtain explicit $P$, $Q$, $\theta$ and $\alpha$ in \eqref{eq:alpha} such that the pole orders of $\alpha$ are minimized. We will accomplish this by descending a suitable purely cubic extension $w^3 = f$ of $K'$.

Recall that by Corollary \ref{cor:allwithclos} there are $2^{t - 1}$ isomorphism classes over $\kbar$ of cubic extensions $L/K$ of the type that we are looking for, and that these are classified by their signed ramification locus up to global negation. Let $X^2 - aX - b$ be the minimal polynomial over $K$ of a generator $r$ of the quadratic extension $K'$ of $K$. On the level of places choosing signs comes down to fixing a decomposition
\begin{equation}\label{eq:Tdecomp}
  \pp_i = \pP_i + \sigma(\pP_i) := \pP_i^- + \pP_i^+
\end{equation}
in $K'$ for every $\pp_i$ in $T$.

\begin{theorem}\label{thm:unifeq}
  Given $K'$ and $T$, as well as decompositions \eqref{eq:Tdecomp}, one can explicitly compute elements $P, Q \in K$ and $\theta = P + Q r \in K$, such that setting $f = \sigma(\theta)/\theta$ and $\alpha = f + f^{-1}$ gives a defining equation $x^3-3x-\alpha$ for $L/K$ as in \eqref{eq:alpha} corresponding to the sign choices in \eqref{eq:Tdecomp}.

  The places where $\alpha$ has a pole are exactly those in $T$; moreover, we have $v_{\pp}(\alpha) \geq~-1$ for all places $\pp \in T$. The only exception occurs if the total degree of $T$ is odd and $S$ consists of a single place of total degree two, in which case there exists one place $\pp \in T$ which has $v_{\pp}(\alpha) = -2$. The twists of $L/K$ can be obtained by replacing $f$ by $u f$ for $u \in N_1/N_1^3$.
\end{theorem}

\begin{proof}[Proof of Theorem \ref{thm:unifeq}]
  We start by observing that if $S$ is non-empty, then $K'$ is not a constant extension, and therefore $K$ has a degree-one place by Remark \ref{rem:PP1base}, which pulls back to a Galois stable divisor $\eta$ of total degree two for $K'$. If $S$ is empty and $K'$ is constant over $K$, then the anticanonical divisor of $K$ similarly defines a Galois stable degree-two divisor $\eta$ for $K'$.

  {\bfseries Case 1: $\boldsymbol{\deg(T)}$ is even.} In this case we write $\deg (T) = 2 e$ and choose $\theta = P + Q r$ for $P, Q \in K$ such that
  \begin{equation}\label{eq:thetadiv}
    (\theta) = \sum_{i = 1}^{t} \pP^-_i - e \cdot \eta .
  \end{equation}
  Let $f = \sigma (\theta) / \theta$. Then $\sigma (f) = 1/f$. Moreover, because $\eta$ is Galois stable we obtain
  \begin{equation}\label{eq:D}
    \left( f \right) = \sum_{i = 1}^{t} \left( \pP^+_i - \pP^-_i\right),
  \end{equation}
  which leads to the right choice of signs.  Let $L' = K' (w)$, where $w^3 = f$. It remains to determine the fixed field of $L'$ under the involution sending $(f, w) \mapsto (1/f, 1/w)$. The element $y = w + w^{-1}$ is invariant under this involution, and we have
  \begin{equation}
    y^3 = w^3 + 3 w + 3 w^{-1} + w^{-3} = 3 y + (w^3 + w^{-3}) = 3 y + (f + f^{-1}) .
  \end{equation}
  We conclude that for $\alpha = f + f^{-1}$ we obtain Eq.~\eqref{eq:alpha}, as claimed. It follows from the construction that $v_{\pp_i}(f) = -1$ for all $i$ and that these are the only poles of $f$, so that $v_{\pp}(\alpha) \geq -1$ for all places $\pp$ of $K$.

  {\bfseries Case 2: $\boldsymbol{\deg(T)}$ is odd}. In this case one of the places in $T$, say $\pp_1$, has odd degree. This means that both $K$ and $K'$ correspond to a projective line, and that both have a place of degree one.

  {\bfseries Case 2a: $\boldsymbol{\deg(T)}$ is odd and $S$ is empty or contains a place of degree one}. Under this hypothesis $K'$ has a Galois stable degree-one place $\pQ$. Indeed, if $S$ is empty, then we can take $\pQ$ to be any degree-one place of $K'$, which exists by the previous paragraph, and if $S$ contains a place $\mathfrak{q}$ of degree one, then we let $\pQ$ be the unique place of $K'$ over it. The construction above can then be adapted by instead considering $\theta = P + Q r$ such that
  \begin{equation}\label{eq:thetadiv2}
    (\theta) = \sum_{i = 1}^{t} \pP^-_i - \deg(T) \cdot \pQ .
  \end{equation}
  Setting $f = \sigma (\theta) / \theta$, we again have \eqref{eq:D}, and as $\alpha = f + f^{-1}$ with $f = \sigma(\theta)/\theta$ we indeed have $v_{\pp}(\alpha) \geq -1$ everywhere.

  {\bfseries Case 2b: $S$ consists of a single place of total degree two}. Theorem \ref{thm:defeq} shows that in order to find a minimized defining equation of the form \eqref{eq:alpha}, we need to find $\theta$ such that \eqref{eq:D} holds. Writing $(\theta) = \sum \pP^-_i - D$ shows that satisfying \eqref{eq:D} requires the existence of a Galois stable divisor $D$ of odd degree on $K'$, yet these do not exist because of the hypothesis on $S$. We conclude that in this case there is no defining polynomial $x^3 - 3 x - \alpha$ with $v (\alpha) \geq -1$ everywhere.

  Instead we let $e = (\deg(T) - 3\deg(\pp_1))/2$ and choose $\theta = P + Q r$ such that
  \begin{equation}\label{eq:thetaodd}
    (\theta) = -3 \pP_1^- + \sum_{i=1}^{t} \pP_i^{-} - e \cdot \eta .
  \end{equation}
  Then for $f = \sigma(\theta)/\theta$ we have that $(f)$ equals \eqref{eq:D} up to $3 (\pP_1^+ - \pP_1^-)$. By Kummer theory, this still yields the correct ramification, and $\alpha = f+f^{-1}$ yields the claimed defining equation.

  For all cases, the fact that the twists are as described is a consequence of Lemma~\ref{lem:twistisos}. Note that by Hilbert's Theorem 90 we may write $u = \sigma(\lambda)/\lambda$. Twisting $f$ to $uf$ is then equivalent to changing $\theta$ to $\lambda \theta$, which can still be written in the form $P'+Q'r$ for $P',Q' \in K$ and which does not affect the divisor~$D$.
\end{proof}

\begin{remark}
  Note that the poles of the element $\theta$ constructed in the proof of Theorem \ref{thm:unifeq} can be controlled if so desired, and that its effective determination is possible via the Riemann--Roch theorem over $k$. Moreover, one can use \eqref{eq:thetadiv2} as long as the ramification locus contains a rational point that is fixed by the involution, for example that given in Remark \ref{rem:ratptinf} for $Q_1$ and $Q$. By choosing a coordinate on the corresponding $\PP^1$ with a pole at the specified point, finding $\theta$ in \eqref{eq:thetadiv2} comes down to expressing the elements $\pp_i$ of $T$ as norms of polynomials. The two possibilities per $\pp_i$ of doing so then give rise to the sign choices in \eqref{eq:Tdecomp}.
\end{remark}

\begin{remark}\label{rem:min}
  We consider Case 2b of Theorem \ref{thm:unifeq} in some more detail. Let $\kappa^-$ be a degree-one place of $K'$, which the proof shows to exist, and let $\kappa^+ = \sigma (\kappa^-)$ be the Galois conjugate of $\kappa^-$. Then there exists a function $\ell$ on $K'$ of degree one such that $\mathrm{div}(\ell) = \kappa^- - \kappa^+$. It follows that $\ell \sigma(\ell) = c$ for some constant $c \in k$. Let $e = (\deg(T) + 1))/2$ and define $\theta = P + Q r$ via
  \begin{equation}
    (\theta) = \kappa^- + \sum_{i = 1}^{t} \pP^-_i - e \eta .
  \end{equation}
  Let $f = \ell \sigma(\theta)/\theta$. Then $\sigma(f) = c/f$, and by Galois stability of $\eta$ we obtain
  \begin{equation}
    \mathrm{div}\left(\ell \frac{\sigma(\theta)}{\theta}\right) = \sum_{i = 1}^{t} \left( \pP^+_i - \pP^-_i\right) .
  \end{equation}
  For $y = w + c w^{-1}$, with $w^3 = f$, one verifies that this gives the generalized defining equation
  \begin{equation}\label{eq:defeqgen}
    y^3 - 3 c y - \alpha = 0
  \end{equation}
  with $\alpha = c(f+ cf^{-1})$, which has at most simple poles at all places of $K$.

  Note that $c$ does not equal $1$, in line with the proof of Theorem \ref{thm:unifeq}. Indeed, after normalizing we may assume that $K = k (x)$, that the extension $K'$ of $K$ is defined by $y^2 = x^2 - d$, and that the points $\kappa^-$ and $\kappa^+$ are over $\infty$. In this case we obtain $\ell = x + y$ and $\ell \sigma (\ell) = (x + y)(x - y) = d$. This cannot be $1$ (or, for that matter, any square), since then the branch locus, defined by $x^2 - d$, would split over~$k$, contrary to our hypothesis that it consists of a single place of degree two.

  The ramification of a generalized defining equation of the form \eqref{eq:defeqgen} can still be read off: Putting $y = w + c w^{-1}$ yields
  \begin{equation}
      w^3 + 3 c w + 3 c^2 w^{-1} + c^3 w^{-3} = 3 c w + 3 c^2 w^{-1} + \alpha,
  \end{equation}
  or $w^3 + c^3 w^{-3} = \alpha$. This means that if we let $s$ be a root of the polynomial $X^2 - \alpha X + c^3$, we have $y = w + c w^{-1}$, where $w$ satisfies $w^3 = s$. Thus the purely cubic closure is defined by $X^2 - \alpha X + c^3$. Therefore, arguing as in \cite{MWcubic} shows that the total ramification of $L/K$ is at the poles of $\alpha$ whose multiplicity is not divisible by three, and that the partial ramification is at the zeros of $\alpha^2 - 4 c^3$ whose multiplicity is not divisible by two.
\end{remark}

\begin{remark}
  The techniques used in this section do not require that the base field $K$ be $k (x)$, as the only crucial hypothesis was that the purely cubic closure $K'$ was a conic. We may equally well take $K$ to be the function field of a non-trivial conic over $k$.
\end{remark}

\section{Bi-twists}\label{sec:biisotwists}

We now consider bi-twists for some classes of impurely cubic extensions. We cannot hope to cover all cases, as this would be equivalent to fully solving the Galois descent problem for superelliptic curves of exponent three. We classify the extensions obtained by the ramification type $R$ of their base extension $L \kbar / K \kbar$. The Riemann--Hurwitz formula yields four possible values for $R$ and the genus $g_L$ of $L$ which we now discuss separately, first for fields $k$ of characteristic unequal to $2$.\\

$\mathbf{(R, g_L) = (3^2, 0)}:$ This case is covered by Theorem \ref{thm:biisog0}. Restricting to conics $Q$ that admit a point will give those bi-twists $Y \to X$ for which the target curve~$X$ is still isomorphic to $\PP^1$ over $K$. We now consider these in some more detail. In Theorem \ref{thm:biisog0} these correspond to the pairs $(Q, D)$ with $Q \cong \PP^1$. Such pairs are classified by the splitting field of the degree-two divisor $D$, and we will construct corresponding twists in the spirit of Theorem \ref{thm:unifeq}. To do this, we choose a quadratic extension $q$ of the base field $k$ along with a defining polynomial $X^2 + a X + b$ for $q$ over $k$.

\begin{proposition}\label{prop:33}
  An equation of the bi-twist of $y^3 = x$ corresponding to $q$ is given by
\begin{equation*}
    y^3 = 3 y + \frac{2 x^2 + 2 a x + (a^2 - 2 b)}{x^2 + a x + b} .
\end{equation*}
\end{proposition}

\begin{proof}
  Consider the roots $r$ and $\rbar$ of the polynomial $X^2 + a X + b$ in $q$. We can take $f = (x - r)/(x + \rbar)$ in Theorem \ref{thm:unifeq} over the constant quadratic extension $q K$. This yields
\begin{equation*}
    \begin{split}
      \alpha & = f + 1/f = \frac{x + \rbar}{x - \rbar} + \frac{x - \rbar}{x + \rbar}
      = \frac{(x - \rbar)^2 + (x - r)^2}{(x - r)(x - \rbar)} \\
      & = \frac{2 x^2 - 2 (r + \rbar) x + r^2 + \rbar^2}{x^2 + a x + b}
        = \frac{2 x^2 + 2 a x + (a^2 - 2 b)}{x^2 + a x + b} .
    \end{split}
\end{equation*}
  Since the ramification locus splits over $q$, this equation gives the requested bi-twist.
\end{proof}

\begin{remark}
  The results above are uniform in that they also function in characteristic $2$. Taking a defining polynomial for $q$ in Artin--Schreier form $X^2 + X + a$ then gives
\begin{equation*}
    \alpha = \frac{1}{x^2 + x + a} .
\end{equation*}
  If the characteristic does not equal $2$, then taking a defining polynomial for $q$ in Kummer form $X^2 - d$ yields
\begin{equation*}
    \alpha = 2 \frac{x^2 + d}{x^2 - d} .
\end{equation*}
\end{remark}

\begin{remark}
  When $k$ is a finite field, the above result shows how to obtain the non-trivial bi-twist mentioned in Corollary~\ref{cor:biisog0}. Its specific defining equation depends on the choice of an irreducible polynomial defining the quadratic extension of $k$.
\end{remark}

$\mathbf{(R, g_L) = (3^3, 1)}:$ This case is covered by Theorem \ref{thm:biisog1}, which also applies when~$k$ is of characteristic $2$. Note that for all these bi-twists $Y \to X$ the base curve~$X$ is indeed isomorphic to $\PP^1$. \\

$\mathbf{(R, g_L) = (3^1\, 2^2, 0)}:$ By Proposition \ref{prop:notwists}, the covers with this ramification do not admit non-trivial twists. However, there is still the possibility of bi-twisting, because there is an automorphism switching the branch points below the partial ramification points. This can be seen by considering the cover
\begin{equation}\label{eq:31220}
  y^3 = 3 y + x,
\end{equation}
which has the given ramification type (namely, full ramification occurring over $x~=~\infty$ and partial ramification over $x = \pm 2$). The relevant bi-automorphism group is $\ZZ / 2 \ZZ$. Therefore, as in the previous case, it suffices to find the bi-twist of the cover \eqref{eq:31220} corresponding to a fixed Galois structure on the branch locus, cut out by the quadratic extension $q$ of the base field defined by the polynomial $X^2 - d$ for instance.

\begin{proposition}\label{prop:322}
  An equation of the bi-twist of $y^3 = 3 y + x$ corresponding to $q$ is given by
\begin{equation*}
    y^3 = 3 y + 2 \frac{2 x^2 - d}{d} .
\end{equation*}
\end{proposition}

\begin{proof}
  We follow Theorem \ref{thm:unifeq}. Let the branch locus be defined by the zeros of $x^2 - d$. The corresponding quadratic extension defined by $u^2 = x^2 - d$ has two rational points over $\infty$. A function with a pole and a zero at these points is given by $(x + u)/(x - u)$. We therefore descend the extension of the projective line with coordinate $x$ defined by
\begin{equation*}
    \left\{
      \begin{array}{ll}
        u^2 & = x^2 - d \\
        w^3 & = \frac{x + u}{x - u}
      \end{array} \right.
\end{equation*}
  under the involution $(u, w) \mapsto (-u, 1/w)$, which yields
\begin{equation*}
    \begin{split}
      \alpha & = f + 1/f = \frac{x + u}{x - u} + \frac{x - u}{x + u}
      = \frac{(x + u)^2 + (x - u)^2}{(x - u)(x + u)} \\
      & = \frac{2 x^2  + 2 u^2}{x^2 - u^2} = 2 \frac{2 x^2 - d}{d} .
    \end{split}
\end{equation*}
\end{proof}

\begin{remark}
  As in the previous case, we obtain two bi-twists when the base field~$k$ is finite, since $k$ then admits a unique quadratic extension $q$. Its specific defining equation depends on the choice of an element $d$ whose square root defines $q$ over $k$.
\end{remark}

$\mathbf{(R, g_L) = (3^2\, 2^2, 1)}:$ We first consider these covers over the algebraic closure $\kbar$. By Proposition~\ref{prop:Qkbar}(1), we may assume that the purely cubic closure is defined by $u^2 = x$, so that the places of partial ramification are over $\infty$ and $0$. We ensure that full ramification takes places over $1$. This gives rise to the versal family of covers defined by descending
\begin{equation}\label{eq:3322up}
  \left\{
    \begin{array}{ll}
      u^2 & = x \\
      w^3 & = \frac{u + 1}{u - 1} \frac{u + \lambda}{u - \lambda}
    \end{array} \right.
\end{equation}
for $\lambda \not\in \left\{ \infty, 0, \pm 1 \right\}$. The resulting descents are given by
\begin{equation}\label{eq:3322down}
  y^3 = y + \frac{x^2 + (\lambda^2 + 4 \lambda + 1) x + \lambda^2}{x^2 - (\lambda^2 + 1) x + \lambda^2} .
\end{equation}

\begin{proposition}\label{prop:3322iso}
  Two covers obtained from \eqref{eq:3322down} for the parameter values $\lambda_1, \lambda_2$ are isomorphic if and only if $\lambda_1 = \lambda_2$ and they have trivial automorphism group over~$\kbar$. They are bi-isomorphic if and only if $\lambda_1 = \lambda_2^{\pm 1}$ and they have bi-automorphism group $\ZZ / 2 \ZZ$ over $\kbar$.
\end{proposition}

\begin{proof}
  By Lemma \ref{lem:pcisodesc}, it suffices to consider the covers \eqref{eq:3322up}. An isomorphism maps the common purely cubic closure to itself and therefore either fixes $u$ or sends it to $-u$. Consideration of the branch locus in~$u$ in the former case shows that $w$ is mapped to $\zeta_3^i w$ for some $i$ and that this does not change the defining equation. Similarly, the latter possibility yields that $w$ is sent to $\zeta_3^i / w$. Once more the defining equation in \eqref{eq:3322up} is not affected. The covers \eqref{eq:3322down} have trivial automorphism group in light of Proposition~\ref{prop:notwists}.

  Now we consider the bi-automorphisms of \eqref{eq:3322up}. By what went before, it suffices to analyse the possible action on the branch locus, which is partitioned into the branch points $\left\{ \infty, 0 \right\}$ with partial ramification upstairs and the branch points $\left\{ 1, \lambda^2 \right\}$ with full ramification upstairs. Generically there is a unique non-trivial automorphism of $\PP^1$ that maps these pairs to themselves, namely the involution $x \mapsto \lambda^2 / x$. This lifts to the bi-automorphism $(x, u, w) \mapsto (\lambda^2 / x, \lambda / u, w)$ of \eqref{eq:3322up}, which commutes with the involution $(x, u, w) \mapsto (x, -u, 1/w)$ used in the descent and therefore defines a bi-automorphism of \eqref{eq:3322down}. The other lifts do not give rise to more bi-automorphisms of \eqref{eq:3322down}.

  It therefore remains to consider the cases where the sets $\left\{ \infty, 0 \right\}$ and $\left\{ 1, \lambda^2 \right\}$ are mapped to themselves by other automorphisms of $\PP^1$. Since the automorphism found previously switches $\infty$ and $0$, we may assume that these points are fixed. The automorphism is then of the form $x \mapsto c x$. A non-trivial such automorphism can only map $\left\{ 1, \lambda^2 \right\}$ to itself if $\lambda = i$, in which case also $c = -1$. The automorphism then lifts by sending $u$ to $i u$. Applying this to the rational fraction
\begin{equation*}
    \frac{u + 1}{u - 1} \cdot \frac{u + i}{u - i}
\end{equation*}
  defining $w^3$ maps it to
\begin{equation*}
    \frac{u + 1}{u - 1} \cdot \frac{u - i}{u + i},
\end{equation*}
  meaning that we do not in fact obtain an automorphism in this case.

  The analysis of this bi-automorphism group also applies when analysing bi-isomorphisms between covers defined by different parameter values $\lambda_1$ and $\lambda_2$. Composing with the non-trivial bi-automorphism described above if needed, we may assume that the isomorphism maps $x \mapsto c x$ for some scalar $c$. Similarly, applying the descent involution $(x, u, w) \mapsto (x, -u, 1/w)$ if necessary, we may assume that the branch locus $\left\{ 1 , \lambda_1 \right\}$ is sent to $\left\{ 1, \lambda_2 \right\}$. This can only happen if $\lambda_1 = \lambda_2^{-1}$, in which case we obtain a bi-isomorphism $(x, u, w) \mapsto (\lambda_2^2 x, \lambda_2 u, w)$ that intertwines the descent involutions and hence induces an isomorphism of the descended covers~\eqref{eq:3322down}.
\end{proof}

We conclude that $\mu = \lambda + \lambda^{-1} \in k \setminus \left\{ \infty, \pm 2 \right\}$ parametrizes the bi-automorphism classes of covers~\eqref{eq:3322down}. Moreover, given a parameter value for $\mu$ in the base field $k$, we can find a corresponding bi-isomorphism class defined over $k$ as follows.

\begin{proposition}
  Given $\mu \in k$ with $\mu \not\in \left\{ \infty, \pm 2 \right\}$, a corresponding cover is defined by
  \begin{equation}\label{eq:3322desc}
    y^3 = 3 y + 2 \frac{x^2 + (\mu + 4) x + 1}{x^2 - \mu x + 1} .
  \end{equation}
  This cover branches partially above $\left\{ \infty, 0 \right\}$ and fully above the roots of $x^2 - \mu x + 1$.
\end{proposition}

\begin{proof}
  This equation can be obtained via the substitution $x \mapsto \lambda x$ in \eqref{eq:3322down}.
\end{proof}

It remains to determine the bi-twists of \eqref{eq:3322desc}. The common non-trivial bi-automorphism is given by $(x, y) \mapsto (1/x, y)$. Substituting $x \to \frac{x + 1}{x - 1}$ and writing the parameter as $\mu = 2 - 4/(1 - \nu)$ yields the alternative family
\begin{equation}\label{eq:altfam}
  y^3 = 3 y + 2 \frac{(2 \nu - 1) x^2 - \nu}{x^2 - \nu}
\end{equation}
with $\nu \not\in \left\{ \infty, 0, 1 \right\}$. This family has partial ramification at $\left\{ \pm 1 \right\}$ and full ramification at the roots of $x^2 - \nu$. The bi-twists of the covers in this family are quickly identified, as follows.

\begin{proposition}\label{prop:3322}
  Given $\nu \in k$ with $\nu \not\in \left\{ \infty, 0, 1 \right\}$, the bi-twists of \eqref{eq:altfam} are parametrized by $k^* / (k^*)^2$. The class defined by an element  $d$ of $k^*$ corresponds to the equation
\begin{equation*}
    y^3 = 3 y + 2 \frac{(2 \nu - 1) x^2 - d \nu}{x^2 - d \nu} .
\end{equation*}
\end{proposition}

\begin{remark}
  As in the previous case, if the base field $k$ is finite, then given $\nu \in k$ with $\nu \not\in \left\{ \infty, 0, 1 \right\}$, we obtain two bi-twists. The defining equation of the non-trivial bi-twist depends on the choice of an element $d$ whose square root defines a quadratic extension $q$ over $k$.
\end{remark}

\begin{remark}
  The family of covers under consideration here (as well as those below for the case $(R, g_L) = (3^2 2^1, 1)$ in characteristic $2$) is related to three-torsion points on the Jacobian of genus-one curves. This is because the corresponding covers have full ramification at two distinct places, which we may assume, after twisting if needed, to be its pole and zero. Note that this is the only ramification possible for a function corresponding to a three-torsion element of a genus-one curve $Y$, except if $j (\Jac (Y)) = 0$, in which case the ramification type $(3^3)$ is also possible (for exactly one of the three-torsion points on the Jacobian of $Y$).
\end{remark}

This concludes our analysis when $\cha(k) \neq 2$.
So now we consider the complementary case $\cha (k) = 2$. The cases $\mathbf{(R, g_L) = (3^2, 0)}$ and $\mathbf{(R, g_L) = (3^3, 1)}$ remain unchanged.\\

$\mathbf{(R, g_L) = (3^1 2^1, 1)}:$ This is the characteristic-$2$ version of the case $(R, g_L) = (3^1 2^2, 1)$ above. A cover with the requested ramification is given by
\begin{equation}\label{eq:3121}
  y^3 = y + x .
\end{equation}
It is fully ramified over $\infty$ and partially ramified over $0$. By Proposition \ref{prop:notwists}, the covers with this ramification do not admit non-trivial automorphisms over $\kbar$.

\begin{proposition}\label{prop:32}
  The cover \eqref{eq:3121} has trivial bi-automorphism group over $\kbar$. In particular, it has no bi-twists.
\end{proposition}

\begin{proof}
  Any bi-automorphism has to fix the distinguished points $\infty$ and $0$ and therefore sends $x$ to $\lambda x$. Moreover, this bi-automorphism would induce a bi-automorphism of the resolvent extension, which is defined by $z^2 = z + x^{-2} + 1$ or alternatively by $z^2 = z + x^{-1} + 1$. But this is impossible by Artin--Schreier theory. Indeed, the existence of such a bi-automorphism implies an isomorphism of the aforementioned extension with $z^2 = z + (\lambda x)^{-1} + 1$, which is impossible unless $\lambda = 1$ since otherwise $(1 - \lambda) x^{-1}$ is not of the form $\alpha^2 + \alpha$.
\end{proof}

$\mathbf{(R, g_L) = (3^2 2^1, 1)}:$ This is the characteristic-$2$ version of the case $(R, g_L) = (3^2\, 2^2, 1)$ above. We first consider these covers over the algebraic closure $\kbar$. We may then assume that the purely cubic closure is defined by $u^2 + u = x$, and can put the place of partial ramification at $\infty$ and a place of full ramification at $0$. This time we cannot scale the other place of full ramification, so we do not fix it. This gives rise to the versal family of covers defined by descending
\begin{equation}\label{eq:332up}
  \left\{
    \begin{array}{ll}
      u^2 + u & = x \\
      w^3 & = \frac{u + 1}{u} \frac{u + \lambda + 1}{u + \lambda}
    \end{array} \right.
\end{equation}
for $\lambda \not\in \left\{ \infty, 0, 1 \right\}$. The resulting covers are partially ramified over $\infty$ and fully ramified over $\left\{ 0, \lambda^2 + \lambda \right\}$. They descend to
\begin{equation}\label{eq:332down}
  y^3 = y + \frac{\lambda^2 + 1}{x^2 + (\lambda^2 + \lambda) x} .
\end{equation}

\begin{proposition}\label{prop:332iso}
  Two covers obtained from \eqref{eq:332down} for different parameter values~$\lambda$ are all distinct up to both isomorphism and bi-isomorphism. They have trivial automorphism group and bi-automorphism group $\ZZ / 2 \ZZ$ over $\kbar$.
\end{proposition}

\begin{proof}
  The first part of Proposition \ref{prop:3322iso} carries over unchanged, except that now we get the possibilities $u \mapsto u$ and $u \mapsto u + 1$. As for the bi-automorphisms of \eqref{eq:332down}, the only non-trivial morphism respecting the branch locus downstairs is $x \mapsto x + \lambda^2 + \lambda$. Up to the descent involution $(x, u, w) \mapsto (x, u + 1, 1/w)$, this lifts to a unique bi-automorphism $(x, u, w) \mapsto (x + \lambda^2 + \lambda, u + \lambda, w)$ of \eqref{eq:332up}, which commutes with the descent involution and therefore defines the unique non-trivial bi-automorphism of \eqref{eq:332down}.

  Because of the existence of this bi-automorphism, any non-trivial bi-isomorphisms between different covers \eqref{eq:332up} would come from a scaling $x \mapsto c x$. As in the proof of Proposition \ref{prop:32}, none of these with $c \neq 1$ lifts to give an automorphism of the common purely cubic closure, and therefore only trivial isomorphisms exist.
\end{proof}

We conclude that $\lambda \in k \setminus \left\{ \infty, 0, 1 \right\}$ parametrizes the bi-automorphism classes of covers \eqref{eq:332down}. In fact, scaling $x$ by $\lambda^2 + \lambda$ gives the simpler family
\begin{equation*}
  y^3 = y + \frac{1}{\lambda^2 (x^2 + x)} .
\end{equation*}
Note the inseparability phenomenon involving $\lambda$. We may replace $\lambda^2$ by $1/\lambda$ since we are over a perfect field, yielding the family
\begin{equation}\label{eq:332desc}
  y^3 = y + \frac{\lambda}{(x^2 + x)} .
\end{equation}
This family has partial ramification at $\left\{ \infty \right\}$ and full ramification at $\left\{ 0, 1 \right\}$. The bi-twists of the covers in this family are quickly identified, as follows.

\begin{proposition}\label{prop:332}
  Given $\lambda \in k$ with $\lambda \not\in \left\{ \infty, 0, 1 \right\}$, the bi-twists of \eqref{eq:332desc} are parametrized by the quadratic extensions of $k$, or alternatively by $k / H$ where as before $H$ is the additive subgroup of $k$ generated by the elements $\gamma^2 + \gamma$ for $\gamma \in k$, and with the class of $a \in k$ corresponding to the equation
\begin{equation*}
    y^3 = y + \frac{\lambda}{(x^2 + x + a)} .
\end{equation*}
\end{proposition}

\begin{proof}
  A cocycle in $H^1 (\Gamma_k, \ZZ / 2 \ZZ) = \Hom (\Gamma_k, \ZZ / 2 \ZZ)$ is determined by the quadratic extension over which it splits. Using the Artin--Schreier exact sequence $0 \to \GGa \to \GGa \to \ZZ / 2 \ZZ \to 0$, we see that applying this cocycle comes down to substituting $x + \lambda$ for $x$, where $\lambda$ satisfies $\lambda^2 + \lambda = a$. This yields the given equation.
\end{proof}

\begin{remark}
  As in the previous case, if the base field $k$ is finite, then given $\lambda \in k$ with $\lambda \not\in \left\{ \infty, 0, 1 \right\}$, we obtain two bi-twists. The defining equation of the non-trivial bi-twist depends on the choice of an element $a$ for which the Artin--Schreier polynomial $t^2 + t + a$ defines $q$ over $k$.
\end{remark}

\begin{remark}
  The following table summarizes the number of bi-isomorphism classes when $k$ is a finite field, and refers to the relevant propositions to determine defining equations. The first pair of entries occur in general characteristic, the second pair in characteristic not equal to $2$, and the final pair in characteristic $2$.
  
  \renewcommand{\arraystretch}{1.5}
  \begin{center}
    \begin{tabular}{|l|l|l|}
      \hline
      $(R, g_L)$ & Number of bi-iso classes & Proposition    \\
      \hline
      $(3^2, 0)$     & $2$         & \ref{prop:33}           \\
      $(3^3, 1)$     & $3$ or $9$  & \ref{thm:biisog1}      \\
      \hline
      $(3^1\, 2^2, 0)$ & $2$         & \ref{prop:322}          \\
      $(3^2\, 2^2, 1)$ & $2 (q - 2)$ & \ref{prop:3322}         \\
      \hline
      $(3^1\, 2^1, 0)$ & $1$         & \ref{prop:32}           \\
      $(3^2\, 2^1, 1)$ & $2 (q - 2)$ & \ref{prop:332}          \\
      \hline
    \end{tabular}
  \end{center}
  Theorems \ref{thm:biisog0} and \ref{thm:biisog1} show that among these classes, the impurely cubic ones are exactly $y^3 = x$ and the cases for $(R, g_L) = (3^3, 1)$. Similarly, the only Galois cases can be the first two with $(R, g_L) = (3^2, 0)$ and $(R, g_L) = (3^3, 1)$. The former always admits a unique Galois bi-twist, and in the latter case either all or none of the extensions is Galois.
\end{remark}

\begin{remark}
  It may be interesting to use Oort--Sekiguchi--Suwa theory \cite{OSS} to study the degeneration of the preceding covers to characteristic $2$. Similarly, some of the above families degenerate to others. For example, take $\cha (k) \neq 2$ and consider
\begin{equation*}
    y^3 = 3 c y + x .
\end{equation*}
  For $c$ non-zero we get $(R, g_L) = (3^1\, 2^2, 0)$, whereas $c = 0$ gives $(R, g_L) = (3^2, 0)$. Similarly, the family of equations
\begin{equation*}
    y^3 = 3 y + 2 \frac{2 x^2 - (\lambda - 1)^2 x + 2 \lambda^2}{(\lambda + 1)^2 x}
\end{equation*}
  has $(R, g_L) = (3^2\, 2^2, 0)$ for $\lambda \neq \pm 1$, whereas $\lambda = 1$ gives $(R, g_L) = (3^2, 0)$.
\end{remark}

\section{Parshin Fibrations}\label{sec:parshin}

In this section we go beyond the constructions above by no longer assuming that the base field $K$ is of genus zero. Let $X$ be a smooth curve over an algebraically closed field $k$ whose characteristic does not divide $6$.

\begin{definition}
  A \emph{Parshin cover} is a degree-three cover $Y \to X$ that ramifies fully over one point of $X$ and that has no further branch points. A \emph{Parshin fibration} is a family of Parshin covers of $X$, such that the collection of the unique branch points of its respective members coincides with the points of $X$.
\end{definition}

\begin{proposition}\label{prop:parshin-prel}
  \begin{enumerate}
    \item The projective line does not admit a Parshin fibration.
    \item Let $\phi : Y \to X$ be a Parshin cover. Then the monodromy group of $\phi$ is isomorphic to $S_3$.
  \end{enumerate}
\end{proposition}

\begin{proof}
  We use the theory of the fundamental group, as we may by \cite[Exposé 13, Corollaire 2.12]{SGA1}. The projective line minus one point has trivial fundamental group, and therefore admits no unramified covers, proving (1). Now let $\phi : Y \to X$ be a Parshin cover of a curve of genus $g > 0$, with full ramification over the point $P$ say. By \emph{loc.\ cit.}, the fundamental group of $X \setminus \left\{ P \right\}$ is
  \begin{equation*}
    \pi_1 := \langle \alpha_1 , \dots, \alpha_g, \beta_1 , \dots, \beta_g, \gamma \; \mid \;  [ \alpha_1, \beta_1 ] \cdots [ \alpha_g, \beta_g ] \gamma = 1 \rangle ,
  \end{equation*}
  and giving a degree-three connected and unramified cover of $X \setminus \left\{ P \right\}$ is equivalent to giving a map $\pi_1 \to S_3$ whose image is transitive. This image, which is the monodromy group of the corresponding cover, is therefore either $A_3$ or $S_3$. Now the ramification type of $P$ in the prolongation $\phi : Y \to X$ of this unramified cover is nothing but the cycle type of the image of $\gamma$ in $S_3$. If the monodromy group were~$A_3$, then all commutators would vanish, so that said cycle type would be trivial. However, we have imposed in our definition of a Parshin cover that this image be a three-cycle. We conclude that indeed the monodromy group of $\phi$ is isomorphic to $S_3$, proving (2).
\end{proof}

Let $\phi : Y \to X$ be a Parshin cover. The Riemann--Hurwitz formula implies that~$Y$ is of genus $3 g (X) - 1$. By Proposition \ref{prop:parshin-prel}, the Galois closure $Z \to X$ of $\phi$ has two ramification points of index three over the branch point of $\phi$, and therefore it is of genus $6 g (X) - 3 = 2 g (Y) - 1$.

We first consider the case where $g (X) = 1$, so that $g (Y) = 2$ and $g (Z) = 3$. The cover $Z \to Y$ is étale, so that $Z$ is hyperelliptic, with automorphism group containing $S_3$. We use the family for this automorphism stratum from \cite{LR11} to calculate the corresponding quotients.

\begin{proposition}\label{prop:par-g1}
  Let
\begin{equation*}
    Z : t^2 = s (s^6 - \lambda s^3 + 1) .
\end{equation*}
  Define $\rho (s, t) = (\zeta_3 s, t)$ and $\sigma (s, t) = (1/s, t/s^4)$. Let
\begin{equation*}
    Y : v^2 = (u^2 - 4) (u^3 - 3 u - \lambda)
\end{equation*}
  and
\begin{equation*}
    X : y^2 = (x^2 - 4) (x - \lambda) .
\end{equation*}
  Define the maps $\psi : Z \to Y$ and $\phi : Y \to X$ by
\begin{equation*}
    (u, v) = \psi (s, t) = (s + s^{-1}, t (1 - s^{-4}) (s + s^{-1})^{-1})
\end{equation*}
  and
\begin{equation*}
    (x, y) = \phi (u, v) = (u^3 - 3 u, v (u^2 - 1)) = (s^3 + s^{-3}, s t (1 - s^{-6})) .
\end{equation*}
  Then $\psi$ is the degree-two quotient morphism of $Z$ by $\sigma$, and $\phi \psi$ is the degree-six quotient of $Z$ by the automorphism group generated by $\rho$ and $\sigma$. Moreover, $\phi$ ramifies triply above the unique point of $X$ at infinity and is unramified elsewhere. It thus furnishes a Parshin cover for $X$.
\end{proposition}

\begin{proof}
  It suffices to verify that the indicated maps are well-defined, invariant under the respective automorphisms, and of the indicated degree, which are direct calculations.
\end{proof}

Since the curves from Proposition \ref{prop:par-g1} range through all $\kbar$-isomorphism classes as $\lambda$ varies, and since we may use translations to put the branch point at infinity, we have found a Parshin fibration over $\kbar$ for the curve $X$ of genus one. There is a similar way to deal with Parshin covers of hyperelliptic curves that ramify over a Weierstrass point:

\begin{proposition}\label{prop:par-g2}
  Let
\begin{equation}\label{eq:par-g2}
    X : y^2 = f (x) = (x^2 - 4) g (x),
\end{equation}
  with $g$ of odd degree, be a hyperelliptic curve. Then the curve $Y$ obtained from the system of equations consisting of \eqref{eq:par-g2} and $z^3 = 3 z - x$, or alternatively the curve defined by
\begin{equation*}
    Y : w^2 = (z^2 - 4) g (z^3 - 3 z),
\end{equation*}
  defines a Parshin cover of $X$ that ramifies over the Weierstrass point of $X$ at infinity. Explicitly, we have
\begin{equation*}
    \begin{split}
      \phi : Y & \to X \\
      (z, w) & \to (z^3 - 3 z, w (z^2 - 1)) .
    \end{split}
\end{equation*}
\end{proposition}

\begin{proof}
  This follows from Theorem \ref{thm:cubclos}: If we take $\alpha = -x$, then the total ramification of the resulting cover occurs at the double pole of $x$, which is the Weierstrass point of $X$ at infinity. Partial ramification occurs at the zeros of $\alpha^2 - 4 = x^2 - 4$ on~$X$ whose valuation is odd, but by construction of $X$ no such zeros exist. The final statement follows by noting that evaluating $x^2 - 4$ at $z^3 - 3 z$ yields $(z^2 - 4)(z^2 - 1)^2$.
\end{proof}

\begin{remark}
  A similar statement to that of Proposition \ref{prop:par-g2} holds for the curves
  \begin{equation}\label{eq:Xquadfac}
    X : y^2 = f (x) = (x^2 - 4 c^3) g (x)
  \end{equation}
  and
\begin{equation*}
    Y : w^2 = (z^2 - 4 c) g (z^3 - 3 c z),
\end{equation*}
  between which there is the map
\begin{equation*}
    \begin{split}
      \phi : Y & \to X \\
      (z, w) & \to (z^3 - 3 c z, w (z^2 - c)) .
    \end{split}
\end{equation*}
  This in fact suffices to find all Parshin covers $\phi$ defined over a given non-algebraically closed base field~$k$ that are ramified over a Weierstrass point of $X$. Indeed, since this Weierstrass branch point is unique, it is defined over $k$ and we may place it at infinity. Moreover, if there exists a Parshin cover of $X$ over $k$, then the defining polynomial of $X$ contains a quadratic factor. Indeed, as is also mentioned below, inspecting the purely cubic closure of $Y \to X$ yields the existence of a degree-two étale cover of $X$, which by uniqueness is defined over $k$ if $\phi$ is. This étale cover $X$ is defined by an equation $z^2 = h$, where $h$ is a function whose divisor is $2$-divisible and which therefore corresponds to a two-torsion element of $X$ that is defined over the base field. This torsion element corresponds to a quadratic factor of the defining polynomial of $X$. Choosing $c$ appropriately, we can therefore find a defining equation of the form \eqref{eq:Xquadfac}, and with it a Parshin cover defined over $k$.
\end{remark}

We conclude by indicating how to construct Parshin covers of genus-two curves that ramify over non-Weierstrass points. Giving closed formulas ceases to be useful, and instead we give an explicit procedure to construct such covers. A corresponding implementation can be found at \cite{parshin-code}. From now on, we start using the fact that $k$ is algebraically closed in an essential manner; in Remark \ref{rem:final} non-algebraically closed base fields will briefly be considered.

So let $X$ be a genus-two curve over $k$. The purely cubic closure of a Parshin cover $\phi : Y \to X$ corresponds to a degree-two étale cover $W \to X$, for which $W$ is again a hyperelliptic curve. Because of our hypothesis $k = \kbar$, we may assume that~$W$ admits a defining equation
\begin{equation}\label{eq:eqwithinv}
  W : v^2 = u^8 + a_3 u^6 + a_2 u^4 + a_1 u^2 + a_0
\end{equation}
and that the map $W \to X$ is the quotient by the involution $i (u, v) = (-u, -v)$, which has no fixed points. We then have
\begin{equation*}
  X : y^2 = x (x^4 + a_3 x^3 + a_2 x^2 + a_1 x + a_0)
\end{equation*}
for $x = u^2$ and $y = u v$. Let $\iota (u, v) = (u, -v)$ be the hyperelliptic involution, and set $j = i \iota = \iota i$.

Let $\widetilde{Q} = (\alpha, \beta)$ be a point on $W$, let $Q = (\alpha^2, \alpha \beta)$ be its image on $X$, and let $\widetilde{Q}' = i (\widetilde{Q}) = (-\alpha, -\beta)$ be the second element of the fibre over $Q$. Let $E = \widetilde{Q} - i (\widetilde{Q})$.

\begin{proposition}\label{prop:tildeP}
  There are exactly two points $\widetilde{P}$ on $W$ such that $3 E \sim i (\widetilde{P}) - \widetilde{P}$, which form a single orbit under $j$.
\end{proposition}

\begin{proof}
  Since $i$ has a quotient of genus two, the subspace of $\Jac (W)$ that is invariant under $i$ is of dimension two, and the anti-invariant complement, that is, the Prym variety of $W \to X$, has dimension one and can be described as the (connected) image of $i - 1$. The class of $3 E$ is in the image by connectedness. (More naively, one can calculate an explicit representation in Mumford coordinates \cite[\S IIIa.2]{mumford-tata2} as in \texttt{universal.m} in~\cite{parshin-code}.)

  We conclude that there exists a point $\widetilde{P}$ with $3 E \sim i (\widetilde{P}) - \widetilde{P}$. If $\widetilde{P}'$ is another such point, then $i (\widetilde{P}) - \widetilde{P} \sim i (\widetilde{P}') - \widetilde{P}'$, hence $i (\widetilde{P}) + \widetilde{P}' \sim i (\widetilde{P}') + \widetilde{P}$ and $\widetilde{P}' = \widetilde{P}$ or $\widetilde{P}' = \iota (i (\widetilde{P})) = j (\widetilde{P})$ by the uniqueness of the linear system $g_2^1$ on a hyperelliptic curve (of degree one in dimension two, provided by the canonical divisor) and the fact that $i$ has no fixed points.
\end{proof}

We will construct a Parshin cover of $X$ that ramifies over the image $P$ of $\widetilde{P}$ on~$X$ by descending a cyclic cover of $W$ that ramifies over $\widetilde{P}$ and $i (\widetilde{P})$. To this end, we follow the method of proof of Theorem~\ref{thm:norms}. By construction, there exists a rational function $f$ on $W$ such that
\begin{equation}\label{eq:fdef}
  (f) = i (\widetilde{P}) + 3 i (\widetilde{Q}) - \widetilde{P} - 3 \widetilde{Q}.
\end{equation}

\begin{theorem}\label{thm:parshin}
  Let $f$ be as \eqref{eq:fdef}, let $f i^* (f) = \lambda \in k$, and let $g = \lambda f$. Then the cover of $X$ defined by
  \begin{equation}\label{eq:parcov}
    z^3 = 3 c z + \alpha
  \end{equation}
  with
\begin{equation*}
    \alpha = g + i^* (g) = g + \lambda^3 g^{-1}=  \lambda (f + \lambda f^{-1})
\end{equation*}
  is a Parshin cover of $X$ that ramifies over $P$.
\end{theorem}

\begin{proof}
  This follows from the proof of Theorem \ref{thm:norms} (whose steps can still be following because we have guaranteed that the divisor involved is principal) along with Remark \ref{rem:min}, which together show that the cover $w^3 = g$ of $W$ descends to \eqref{eq:parcov}. This latter cover has purely cubic closure $W$ by construction and therefore does not have any partial ramification. Its total ramification occurs over the fibre of $P$, which contains the points $i (\widetilde{P})$ and $\widetilde{P}$ where the function $g$ has a valuation that is not a multiple of three.
\end{proof}

The techniques above furnish a Parshin fibration of $X$ when varying the point $\widetilde{Q} \in W$. Alternatively, one applies the construction for the canonical point $(u, v)$ of~$W$ over its function field $k (W) = k (u, v)$.

\begin{example}
  Consider the hyperelliptic curve
\begin{equation*}
    X : y^2 = x (x^4 + 4 x^3 + 4 x^2 - 5)
\end{equation*}
  over the base field $k = \QQ$. It has the rational point $Q = (1, 2)$, as well as an étale cover
\begin{equation*}
    W : y^2 = x^8 + 4 x^6 + 4 x^4 - 5
\end{equation*}
  defined over $\QQ$. On $W$, the point $Q$ lifts to $\widetilde{Q} = (1, 2)$ and $\widetilde{Q}' = (-1, -2)$. Using the Mumford representation with respect to the divisor of $X$ at infinity, the divisor $\widetilde{Q} - i (\widetilde{Q})$ corresponds to the pair $(x^2 - 1, 2)$. Multiplying this with $3$, we obtain the Mumford coordinates $(x^2 - 49/9, 3278/81)$, which corresponds to the divisor $i (\widetilde{P}) - \widetilde{P}$ for $\widetilde{P} = (7/3, -3278/81)$, which has image $P = (49/9, -22946/243)$ on $X$. Using our construction, we find a Parshin cover of $X$ which ramifies over $P$. It is defined by
\begin{equation*}
    z^3 = 15 z + \alpha
\end{equation*}
  with
\begin{equation*}
    \alpha = \frac{(-320 x - 480) y + (210 x^4 + 1320 x^3 - 6860 x^2 + 2680 x + 1050)}{9 x^4 - 76 x^3 + 174 x^2 - 156 x + 49} .
\end{equation*}
\end{example}

\begin{remark}\label{rem:final}
  There remain considerable rationality obstructions when extending the above procedure to the case where the base field $k$ is not algebraically closed. Firstly, we may then not be able to write the étale cover $W$ of $X$ in the form \eqref{eq:eqwithinv}. Secondly, finding $\widetilde{P}$ in Proposition \ref{prop:tildeP} may require a further quadratic extension. Finally, while we have constructed $\widetilde{P}$ starting from $\widetilde{Q}$, it is more difficult to determine $\widetilde{Q}$ given $\widetilde{P}$, as this essentially amounts to $3$-divisibility of the class of $i (\widetilde{P}) - \widetilde{P}$ in the Jacobian of $W$. We have explored these matters to some small extent in \cite{parshin-code}, but leave more precise considerations for future work.
\end{remark}

\section*{Acknowledgments}
We thank Kiran Kedlaya for bringing the problem of determining the covers in Section \ref{sec:parshin} to our attention, a referee for pointing us to the reference \cite{weir-thesis}, and another referee for their careful reading and suggestions.


\begin{thebibliography}{9}

\bibitem{belabas}
Karim Belabas, \emph{A fast algorithm to compute cubic fields}, Math. Comp.
  \textbf{66} (1997), no.~219, pp.~1213--1237.

\bibitem{cohenetal}
  Henri Cohen, Simon Rubinstein-Salzedo, and Frank Thorne, \emph{Identities for
  field extensions generalizing the {O}hno-{N}akagawa relations}, Compos. Math.
  \textbf{151} (2015), no.~11, pp.~2059--2075.

\bibitem{pohst2}
Claus Fieker, Istv{\'a}n Ga\'{a}l, and Michael Pohst, \emph{On computing
  integral points of a {M}ordell curve over rational function fields in
  characteristic {$>3$}}, J. Number Theory \textbf{133} (2013), no.~2,
 pp.~738--750.

\bibitem{SGA1}
Alexander Grothendieck, \emph{Rev{\^e}tements {\'e}tales et groupe fondamental
  ({S}{G}{A} 1)}, Lecture notes in mathematics, vol. 224, Springer-Verlag,
  1971.

\bibitem{scheidler1}
Michael Jacobson~Jr., Yoonjin Lee, Renate Scheidler, and Hugh Williams,
  \emph{Construction of all cubic function fields of a given square-free
  discriminant}, Int. J. Number Theory \textbf{11} (2015), no.~6, pp.~1839--1885.

\bibitem{parshin-code}
Valentijn Karemaker, Sophie Marques, and Jeroen Sijsling, \emph{{\tt
  parshin\_ex\-pe\-ri\-ments}, some \textsc{Magma} experiments with {P}arshin
  covers}, \href{https://github.com/JRSijsling/parshin\_experiments}{{\tt 
  https://github.com/JRSijsling/parshin\_ex\-pe\-ri\-ments}}, 2020.

\bibitem{scheidler4}
Eric Landquist, Pieter Rozenhart, Renate Scheidler, Jonathan Webster, and
  Qingquan Wu, \emph{An explicit treatment of cubic function fields with
  applications}, Canad. J. Math. \textbf{62} (2010), no.~4, pp.~787--807.

\bibitem{LR11}
Reynald Lercier and Christophe Ritzenthaler, \emph{Hyperelliptic curves and
  their invariants: geometric, arithmetic and algorithmic aspects}, J. Algebra
  \textbf{372} (2012), pp.~595--636.

\bibitem{LRS-Explicit}
Reynald Lercier, Christophe Ritzenthaler, and Jeroen Sijsling, \emph{Explicit
  {G}alois obstruction and descent for hyperelliptic curves with tamely cyclic
  reduced automorphism group}, Math. Comp. \textbf{85} (2016), no.~300,
  pp.~2011--2045.

\bibitem{MWcubic3}
Sophie Marques and Jacob Ward, \emph{A complete study of the ramification for
  any separable cubic global function field}, Res. Number Theory \textbf{5}
  (2019), no.~4, Art. 36, 17.

\bibitem{MWcubic}
Sophie Marques and Kenneth Ward, \emph{A complete classification of cubic
  function fields over any finite field}, arXiv:1612.03534, 2017.

\bibitem{MWcubic2}
Sophie Marques and Kenneth Ward, \emph{Cubic fields: a primer}, Eur. J. Math. \textbf{5} (2019), no.~2,
  pp.~551--570.

\bibitem{mumford-tata2}
David Mumford, \emph{{Tata lectures on theta. II: Jacobian theta functions and
  differential equations. With the collaboration of C. Musili, M. Nori, E.
  Previato, M. Stillman, and H. Umemura. Reprint of the 1984 edition.}},
  {Modern Birkh\"auser Classics. Basel: Birkh\"auser. xiv, 272~p.
  EUR~34.90/net; SFR~59.90 }, 2007.

\bibitem{parshin}
  Jennifer S.\ Balakrishnan, Alex J.\ Best, Francesca Bianchi, Brian Lawrence,
  J.\ Steffen Müller, Nicholas Triantafillou, and Jan Vonk, \emph{Two recent
    $p$-adic approaches towards the (effective) Mordell conjecture},
    arXiv:1910.12755, 2020.

\bibitem{pohst1}
Michael Pohst, \emph{On computing non-{G}alois cubic global function fields of
  prescribed discriminant in characteristic {$>3$}}, Publ. Math. Debrecen
  \textbf{79} (2011), no.~3-4, pp.~611--621.

\bibitem{rozenhart2}
Pieter Rozenhart, Michael Jacobson, and Renate Scheidler, \emph{Tabulation of
  cubic function fields via polynomial binary cubic forms}, Math. Comp.
  \textbf{81} (2012), no.~280, pp.~2335--2359.

\bibitem{rozenhart1}
Pieter Rozenhart and Renate Scheidler, \emph{Tabulation of cubic function
  fields with imaginary and unusual {H}essian}, Algorithmic number theory,
  Lecture Notes in Comput. Sci., vol. 5011, Springer, Berlin, 2008,
  pp.~357--370.

\bibitem{scheidler3}
Renate Scheidler, \emph{Algorithmic aspects of cubic function fields},
  Algorithmic number theory, Lecture Notes in Comput. Sci., vol. 3076,
  Springer, Berlin, 2004, pp.~395--410.

\bibitem{scheidler2}
Renate Scheidler, \emph{Construction of all cubic fields of a fixed fundamental
  discriminant}, Cubic Fields With Geometry (S. Hambleton \& H. C. Williams),
  Springer International Publishing, New York, 2018, pp.~173--302.

\bibitem{OSS}
Tsutomu Sekiguchi, Frans Oort, and Noriyuki Suwa, \emph{On the deformation of
  {A}rtin-{S}chreier to {K}ummer}, Ann. Sci. \'{E}cole Norm. Sup. (4)
  \textbf{22} (1989), no.~3, pp.~345--375.

\bibitem{Serre-Topics}
Jean-Pierre Serre, \emph{Topics in {G}alois theory}, second ed., Research Notes
  in Mathematics, vol.~1, A K Peters, Ltd., Wellesley, MA, 2008, With notes by
  Henri Darmon.

\bibitem{shanks-cubic}
Daniel Shanks, \emph{Class numbers of the simplest cubic fields}, Math. Comp.
  \textbf{1} (1974), pp.~1137--1152.

\bibitem{Vil}
Gabriel Villa-Salvador, \emph{Topics in the theory of algebraic function
  fields}, Birkh\"{a}user, 2006.

\bibitem{weir}
Colin Weir, Renate Scheidler, and Everett Howe, \emph{Constructing and
  tabulating dihedral function fields}, A{NTS} {X}---{P}roceedings of the
  {T}enth {A}lgorithmic {N}umber {T}heory {S}ymposium, Open Book Ser., vol.~1,
  Math. Sci. Publ., Berkeley, CA, 2013, pp.~557--585.

\bibitem{weir-thesis}
  Colin Weir, \emph{Constructing and Tabulating Dihedral Function Fields}, PhD
  thesis, University of Calgary, 2013.

\bibitem{xarles-towers}
Xavier Xarles, \emph{Trivial points on towers of curves}, J. Th\'{e}or. Nombres
  Bordeaux \textbf{25} (2013), no.~2, pp.~477--498.

\end{thebibliography}
\end{document}